\documentclass[11pt]{amsart}
\usepackage[foot]{amsaddr}
\usepackage[english]{babel}

\allowdisplaybreaks

\usepackage[%pagebackref,
pdftex,hyperfootnotes]{hyperref}
\hypersetup{
	colorlinks=true,
	linkcolor=NavyBlue, 
	urlcolor=RoyalPurple,
	citecolor=OliveGreen,
	pdftitle={Hypergeometric discrete multiple orthogonal polynomials},
	bookmarks=true,
	pdfpagemode=FullScreen,
}
\usepackage[usenames,dvipsnames,svgnames,table,x11names]{xcolor}
\usepackage{pgfplots}
\usepackage{tikz}
\usepackage{tikz-3dplot}
\usetikzlibrary{automata,quotes, chains,matrix,calc,shadows,shapes.callouts,shapes.geometric,shapes.misc,positioning,patterns,decorations.shapes,
	decorations.pathmorphing,decorations.markings,decorations.fractals,decorations.pathreplacing,shadings,fadings,arrows.meta,bending}
\usepackage{nicematrix}
\usepackage[utf8]{inputenc}

\usepackage{comment}
\usepackage[textwidth=17.5cm,textheight=22.275cm,
height=%24cm,
24.275cm,
width=18cm]{geometry}

\usepackage{amssymb,latexsym,amsmath,amsthm,bm}
\usepackage{mathrsfs}
\usepackage{mathtools,arydshln,mathdots}

\usepackage{tcolorbox}

\usepackage[table]{xcolor}

\usepackage{Baskervaldx}
\usepackage[]{newtxmath}

\usepackage{bigints}

\theoremstyle{plain}

\newtheorem{teo}{Theorem}[section]
\newtheorem{coro}[teo]{Corollary}
\newtheorem{lemma}[teo]{Lemma}
\newtheorem{pro}[teo]{Proposition}

\theoremstyle{defi}
\newtheorem{defi}[teo]{Definition}

\theoremstyle{remark}
\newtheorem{rem}[teo]{Remark}

\renewcommand{\d}{\operatorname{d}}
\newcommand{\Exp}[1]{\operatorname{e}^{#1}}

\newcommand{\diag}{\operatorname{diag}}

\newcommand{\C}{\mathbb{C}}

\newcommand{\N}{\mathbb{N}}

\renewcommand{\intercal}{{\mathsf{\scriptscriptstyle T}}}

\allowdisplaybreaks[1]

\makeatletter
\DeclareRobustCommand{\gaussk}{\DOTSB\gaussk@\slimits@}
\newcommand{\gaussk@}{\mathop{\vphantom{\sum}\mathpalette\bigcal@{K}}}

\newcommand{\bigcal@}[2]{%
	\vcenter{\m@th
		\sbox\z@{$#1\sum$}%
		\dimen@=\dimexpr\ht\z@+\dp\z@
		\hbox{\resizebox{!}{0.8\dimen@}{$\mathcal{K}$}}%
	}%
}
\newcommand{\cfracplus}{\mathbin{\cfracplus@}}
\newcommand{\cfracplus@}{%
	\sbox\z@{$\dfrac{1}{1}$}%
	\sbox\tw@{$+$}%
	\raisebox{\dimexpr\dp\tw@-\dp\z@\relax}{$+$}%
}
\newcommand{\cfracdots}{\mathord{\cfracdots@}}
\newcommand{\cfracdots@}{%
	\sbox\z@{$\dfrac{1}{1}$}%
	\sbox\tw@{$+$}%
	\raisebox{\dimexpr\dp\tw@-\dp\z@\relax}{$\cdots$}%
}
\makeatother

\makeatletter
\newcommand*{\relrelbarsep}{.386ex}
\newcommand*{\relrelbar}{%
	\mathrel{%
		\mathpalette\@relrelbar\relrelbarsep
	}%
}
\newcommand*{\@relrelbar}[2]{%
	\raise#2\hbox to 0pt{$\m@th#1\relbar$\hss}%
	\lower#2\hbox{$\m@th#1\relbar$}%
}
\providecommand*{\rightrightarrowsfill@}{%
	\arrowfill@\relrelbar\relrelbar\rightrightarrows
}
\providecommand*{\leftleftarrowsfill@}{%
	\arrowfill@\leftleftarrows\relrelbar\relrelbar
}
\providecommand*{\xrightrightarrows}[2][]{%
	\ext@arrow 0359\rightrightarrowsfill@{#1}{#2}%
}
\providecommand*{\xleftleftarrows}[2][]{%
	\ext@arrow 3095\leftleftarrowsfill@{#1}{#2}%
}
\makeatother

\newcommand*\pFqskip{8mu}
\catcode`,\active
\newcommand*\pFq{\begingroup
	\catcode`\,\active
	\def ,{\mskip\pFqskip\relax}%
	\dopFq
}
\catcode`\,12
\def\dopFq#1#2#3#4#5{%
	{}_{#1}F_{#2}\biggl[\genfrac..{0pt}{}{#3}{#4};#5\biggr]%
	\endgroup
}

\allowdisplaybreaks[1]
\begin{document}
\title[Discrete hypergeometric multiple orthogonality]{ Toda  and Laguerre--Freud equations and tau functions for \\ hypergeometric discrete multiple orthogonal polynomials}

%\subtitle{	A
%	Riemann-Hilbert Problem perspective}
%
\author[I Fernández-Irisarri]{Itsaso Fernández-Irisarri$^1$}
\email{$^1$itsasofe@ucm.es}
\address{$^{1,2}$Departamento de Física Teórica, Universidad Complutense de Madrid, Plaza Ciencias 1, 28040-Madrid, Spain}

\author[M Mañas]{Manuel Mañas$^2$}
\email{$^2$manuel.manas@ucm.es}

%\address{$^2$Departamento de Física Teórica, Universidad Complutense de Madrid, Plaza Ciencias 1, 28040-Madrid, Spain \&
%	Instituto de Ciencias Matematicas (ICMAT), Campus de Cantoblanco UAM, 28049-Madrid, Spain}

\begin{abstract}
In this paper, the authors investigate the case of discrete multiple orthogonal polynomials with two weights on the step line, which satisfy Pearson equations. The discrete multiple orthogonal polynomials in question are expressed in terms of $\tau$-functions, which are \emph{double} Wronskians of generalized hypergeometric series. The shifts in the spectral parameter for type II and type I multiple orthogonal polynomials are described using banded matrices. It is demonstrated that these polynomials offer solutions to multicomponent integrable extensions of the nonlinear Toda equations. Additionally, the paper characterizes extensions of the Nijhoff--Capel totally discrete  Toda equations. The hypergeometric $\tau$-functions are shown to provide solutions to these integrable nonlinear equations. Furthermore, the authors explore Laguerre-Freud equations, nonlinear equations for the recursion coefficients,  with a particular focus on the multiple Charlier, generalized multiple Charlier, multiple Meixner II, and generalized multiple Meixner II cases.
\end{abstract}


\maketitle

%\newpage
\tableofcontents
\allowdisplaybreaks
\section{Introduction}

Discrete multiple orthogonal polynomials, $\tau$-functions, generalized hypergeometric series, Pearson equations, Laguerre--Freud equations, and Toda type equations are intriguing mathematical objects that have attracted considerable attention due to their deep connections with various areas of mathematics and physics. These intertwined concepts play a fundamental role in understanding the underlying structures and dynamics of discrete systems, integrable models, and special functions. In this paper, we delve into the rich mathematical theory of discrete multiple orthogonal polynomials, explore their relationship with $\tau$-functions and generalized hypergeometric series, and investigate their applications in Laguerre--Freud equations and Toda type equations.

Orthogonal polynomials \cite{Ismail, Chihara, Beals} are a fascinating and powerful tool in Applied Mathematics and Theoretical Physics, with a wide range of applications in various fields. They possess remarkable properties and have been extensively studied for their applications in approximation theory, numerical analysis, and many other areas. Among the vast family of orthogonal polynomials, two special classes that have attracted significant attention are multiple orthogonal polynomials and discrete orthogonal polynomials.

Back in 1895, Karl Pearson \cite{Pearson}, in the context of fitting curves to real data, introduced his famous family of frequency curves by means of the differential equation
$\frac{f'(x)}{f (x)} = \frac{p_1(x)}{p_2(x)}$, 
where $f$ is the probability density and $p_i$ is a polynomial in $x$ of degree at most $i$, $i = 1, 2$. Since then, a vast bibliography has been developed regarding the properties of Pearson distributions. The connection between Pearson equations and orthogonal polynomials lies in the fact that the solutions to Pearson equations often coincide with or are closely related to certain families of orthogonal polynomials. They are characterized by a weight function, which determines the orthogonality properties of the associated orthogonal polynomials. These equations have been extensively studied due to their close relationship with special functions, such as the classical orthogonal polynomials (e.g., Legendre, Hermite, Laguerre) and their generalizations. In this work, we will use the Gauss--Borel approach to orthogonal polynomials, for a review paper see \cite{intro}. In previous papers, we have explored the application of this technique to the study of discrete hypergeometric orthogonal polynomials, Toda equations, $\tau$-functions, and Laguerre--Freud equations, see \cite{Manas_Fernandez-Irisarri, Fernandez-Irrisarri_Manas, Fernandez-Irrisarri_Manas_2}.

Multiple orthogonal polynomials \cite{nikishin_sorokin, Ismail, andrei_walter, afm} arise when multiple weight functions are involved in the orthogonality conditions. Unlike classical orthogonal polynomials, where a single weight function is used, multiple orthogonal polynomials exhibit a more intricate structure due to the interplay of multiple weight functions. These polynomials have been studied extensively and find applications in areas such as simultaneous approximation, random matrix theory, and statistical mechanics. For the appearance and use of the Pearson equation in multiple orthogonality, see \cite{Coussment, bfm}, see also \cite{bfm2}. On the other hand, discrete orthogonal polynomials \cite{Ismail, Nikiforov_Suslov_Uvarov} emerge when the orthogonality conditions are defined on a discrete set of points. These polynomials find applications in various discrete systems, such as combinatorial optimization, signal processing, and coding theory. Discrete orthogonal polynomials provide a valuable tool for analyzing and solving problems in discrete settings, where the underlying structure is often characterized by a finite or countable set of points. Discrete Pearson equations play an important role in the classification of classical discrete orthogonal polynomials \cite{Nikiforov_Suslov_Uvarov}. When a discrete Pearson equation is fulfilled by the weight, we are dealing with semiclassical discrete orthogonal polynomials, see \cite{diego_paco, diego_paco1, diego, diego1}. Multiple discrete orthogonal polynomials were discussed in \cite{Arvesu}.

Integrable discrete equations \cite{Hietarinta} have emerged as a fascinating and important field of study in Mathematical Physics, providing insights into the behavior of discrete dynamical systems with remarkable properties. These equations possess a rich algebraic and geometric structure, allowing for the existence of soliton solutions, conservation laws, and integrability properties. The concept of integrability in discrete equations goes beyond mere solvability and extends to the existence of an abundance of conserved quantities and symmetries \cite{Adler}. This integrability property enables the development of powerful mathematical methods for studying and understanding their behavior. One of the key features of integrable discrete equations is their connection to the theory of orthogonal polynomials. These polynomials play a central role in the construction of discrete equations with special properties. The link between discrete equations and orthogonal polynomials provides a deep understanding of their solutions, recursion relations, and symmetry properties.

The $\tau$-function \cite{harnad} is a key concept in the theory of integrable systems and serves as a bridge between the spectral theory of linear operators and the dynamics of soliton equations. It has profound connections with algebraic geometry, representation theory, and special functions.
Semiclassical orthogonal polynomials, along with their corresponding functionals supported on complex curves, have been extensively studied in relation to isomonodromic tau functions and matrix models by Bertola, Eynard, and Harnad \cite{bertola}.
 In the context of discrete multiple orthogonal polynomials, the $\tau$-function captures the underlying integrable structure and provides insights into the dynamics and symmetries of the corresponding discrete systems. Its study enables us to uncover deep connections between different mathematical objects and uncover hidden patterns.

Generalized hypergeometric series, another important topic in this research, constitute a class of special functions that arise in diverse mathematical contexts \cite{Hipergeometricos, LibrodeHypergeom}. They possess remarkable properties and have been extensively studied due to their connections with combinatorics, number theory, and mathematical physics. Generalized hypergeometric series play a key role in expressing solutions of differential equations, evaluating integrals, and understanding the behavior of various mathematical models. They have a pivotal role in the Askey scheme that gathers together different important families of orthogonal families within a chain of limits \cite{Koekoek}. See \cite{AskeyII} for an extension of the Askey scheme to multiple orthogonality. Recently, we have completed a part of the multiple Askey scheme by finding hypergeometric expressions for the Hahn type I multiple orthogonal polynomials and their descendants in the multiple Askey scheme \cite{BDFM} (but for the multiple Hermite case).

Laguerre--Freud equations are a family of differential equations that emerge in the study of orthogonal polynomials associated with exponential weights. For orthogonal polynomials, these Laguerre--Freud equations represent nonlinear relations for the recursion coefficients in the three-term recurrence relation. In other words, we are dealing with a set of nonlinear equations for the coefficients $\{\beta_n,\gamma_n\}$ of the three-term recursion relations $zP_n(z) = P_{n+1}(z) + \beta_nP_n(z) + \gamma_nP_{n-1}(z)$ satisfied by the orthogonal polynomial sequence. These Laguerre--Freud equations take the form $\gamma_{n+1} = \mathcal{G}(n, \gamma_n, \gamma_{n-1}, \dots, \beta_n, \beta_{n-1},\dots)$ and $\beta_{n+1} = \mathcal{B}(n,\gamma_{n+1},\gamma_n,\dots,\beta_n,\beta_{n-1},\dots)$. Magnus \cite{magnus,magnus1,magnus2,magnus3}, referring to \cite{laguerre,freud}, named these types of relations Laguerre--Freud relations. Several papers discuss Laguerre--Freud relations for the generalized Charlier, generalized Meixner, and type I generalized Hahn cases \cite{smet_vanassche,clarkson,filipuk_vanassche0,filipuk_vanassche1,filipuk_vanassche2,diego}.

Laguerre--Freud equations are also connected to Painlevé equations. The Painlevé equations are a set of nonlinear ordinary differential equations that were first introduced by the French mathematician Paul Painlevé in the early 20th century. These equations have attracted significant attention due to their integrability properties and the rich mathematical structures associated with their solutions. The connection between orthogonal polynomials and Painlevé equations lies in the fact that certain families of orthogonal polynomials can be related to solutions of specific Painlevé equations. This relationship provides deep insights into the solvability and analytic properties of the Painlevé equations and allows for the construction of explicit solutions using the theory of orthogonal polynomials. For an account of the connection between these fields of study, see \cite{Van Assche,clarkson,clarkson2}.

Moreover, building upon the mentioned studies \cite{Manas_Fernandez-Irisarri, Fernandez-Irrisarri_Manas, Fernandez-Irrisarri_Manas_2}, this paper investigates the relationship between discrete multiple orthogonal polynomials, $\tau$-functions, generalized hypergeometric series, and specific equations such as Laguerre--Freud equations and Toda type equations.

Toda type equations, on the other hand, are nonlinear partial differential-difference equations with wide-ranging applications in mathematical physics and integrable systems \cite{harnad,Hietarinta}. Exploring the connections between these equations and the aforementioned mathematical objects sheds light on the interplay between discrete systems, special functions, and integrable models.

In this paper, we aim to provide a comprehensive overview of the theory of discrete multiple orthogonal polynomials, $\tau$-functions, generalized hypergeometric series, Laguerre--Freud equations, and Toda type equations. We will delve into the mathematical properties, explicit representations, and recurrence relations of discrete multiple orthogonal polynomials. Furthermore, we will explore the relationship between these polynomials, $\tau$-functions, and generalized hypergeometric series. Finally, we will investigate the applications of these concepts to Laguerre--Freud equations and Toda type equations, including both continuous and totally discrete types, showcasing their relevance in the field of mathematical physics and integrable systems.

The layout of the paper is as follows. We continue this section with a basic introduction to discrete multiple orthogonal polynomials and their associated tau functions. In the second section, we delve into Pearson equations for multiple orthogonal polynomials, generalized hypergeometric series, and the corresponding hypergeometric discrete multiple orthogonal polynomials. We discuss their properties, including the Laguerre--Freud matrix that models the shift in the spectral variable, as well as contiguity relations and their consequences. 

In the third section, we present two of the main findings discussed in this paper. In Theorem \ref{theo:third_order_PDE}, we present nonlinear partial differential equations of the Toda type, specifically of the third order, for which the hypergeometric tau functions provide solutions. Additionally, in Theorem \ref{rem:tau_Nijhoff_Capel}, we provide a completely discrete system, extending the completely discrete Nikhoff--Capel Toda equation, and its solutions in terms of the hypergeometric tau functions. 

Finally, in the fourth section, we present explicit Laguerre--Freud equations for the first time for the multiple versions of the generalized Charlier and generalized Meixner II orthogonal polynomials.

\subsection{Discrete multiple orthogonality on the step line with two weights}
Let us introduce the following vectors of monomials in the independent variable $x$
\begin{align*}X&=\begin{bNiceMatrix}1&x&x^2&\Cdots[shorten-end=7pt]\end{bNiceMatrix}^\intercal,& 
	X^{(1)}&=\begin{bNiceMatrix}1&0&x&0&\Cdots[shorten-end=7pt,shorten-start=5pt]\end{bNiceMatrix}^\intercal, & X^{(2)}&=\begin{bNiceMatrix}0&1&0&x&\Cdots[shorten-end=7pt,shorten-start=5pt]\end{bNiceMatrix}^\intercal. 
\end{align*}
Let us also consider a couple of weights $w^{(1)},w^{(2)}:\N_0\to\C$  defined
on the homogeneous lattice $\N_0$.  The corresponding moment matrix is given by 
\begin{align*} 
	\mathscr M&\coloneq \sum^{\infty}_{k=0}X (k)\big(X^{(1)} (k)w^{(1)}(k)+X^{(2)} (k)w^{(2)}(k)\big)^{\intercal}.
\end{align*}
%where $\xi\coloneq X^{(1)} w^{(1)}+X^{(2)} w^{(2)}$.
We also consider the matrices
\begin{align*}
	I^{(1)}&\coloneq\diag(1,0,1,0,\dots), & 	I^{(2)}&\coloneq\diag(0,1,0,1,\dots).
\end{align*} 
and  shift matrices
\begin{align*}
\Lambda\coloneq	\begin{bNiceMatrix}[columns-width=auto]
		0 & 1 & 0&\Cdots[shorten-end=10pt]&\\
		0& 0&1 &\Ddots[shorten-end=10pt]&\\
		\Vdots[shorten-start=5pt,shorten-end=7pt]&\Ddots[shorten-end=-22pt]&\Ddots[shorten-end=-13pt]& \Ddots[shorten-end=9pt]&\\
		&&& &
	\end{bNiceMatrix},
\end{align*}
and $\Lambda^{(a)}\coloneq \Lambda^2 I^{(a)}$, $a\in\{1,2\}$.
The following equations are satisfied
\begin{align*}
	\Lambda X(x)&=xX(x),&
		 \Lambda^{(1)}X^{(1)}(x)&=xX^{(1)}(x),& \Lambda^{(2)}X^{(2)}(x)&=xX^{(2)}(x),
&	\Lambda^{(1)}X^{(2)}(x)&=0,&\Lambda^{(2)}X^{(1)}(x)&=0.
\end{align*}
The  moment matrix satisfies 
\begin{align*}
	\Lambda \mathscr M= \mathscr M(\Lambda^\intercal)^2,
\end{align*}
so that we say that a matrix is bi-Hankel matrix.

The Gauss-Borel  or $LU$ factorization of the moment matrix is
\begin{align}\label{eq:LU}
	 \mathscr M=S^{-1}H\tilde{S}^{-\intercal}
\end{align}
where $S$ and $\tilde{S}$ are lower unitriangular matrices and $H$ is a diagonal matrix.

We define the truncated matrix $\mathscr M^{[n]}$ as the $n$-th leading principal submatrix of the moment matrix $
\mathscr M$, i.e. built up with the first $n$ rows and columns. Then, we introduce the so called $\tau$-function, as
\begin{align}\label{eq:tau}
	\tau_n=\det\mathscr M^{[n]}.  
\end{align}
Factorization \eqref{eq:LU} exists whenever $\tau_n\neq 0$ for $n\in\N_0$.

 Matrices $S$ and $\tilde{S}$  expand as power series in $\Lambda^\intercal$ with diagonal matrices coefficients
\begin{gather*}
\begin{aligned}
	S&=I+\Lambda^\intercal S^{[1]}+(\Lambda^\intercal)^2 S^{[2]}+\cdots, & S^{-1}&=I+\Lambda^\intercal S^{[-1]}+(\Lambda^\intercal)^2 S^{[-2]}+\cdots,
\end{aligned} \\ \begin{aligned}
\tilde{S}&=I+\Lambda^\intercal \tilde{S}^{[1]}+(\Lambda^\intercal)^2 \tilde{S}^{[2]}+\cdots,& \tilde{S}^{-1}&=I+\Lambda^\intercal \tilde{S}^{[-1]}+(\Lambda^\intercal)^2 \tilde{S}^{[-2]}+\cdots.
\end{aligned}
\end{gather*}
In particular,  the following equations  hold
\begin{align*}
S^{[-1]}&=-S^{[1]}, &
S^{[-2]}&=-S^{[2]}+(\mathfrak a_{-}S^{[1]})S^{[1]},
\end{align*}
and matrices $\tilde{S}^{[-1]}$ and $\tilde{S}^{[-2]}$ fulfill analogous relations.

In  what follows we will use  lowering and rising shift  matrices given by 
\begin{align*}
	\mathfrak a_{-}\diag{(m_0,m_1,\dots)}&=\diag{(m_1,m_2,\dots)}, & \mathfrak a_{+}\diag{(m_0,m_1,\dots)}&=\diag{(0,m_0,\dots)}.
\end{align*}
Any diagonal matrix $D$ fulfills the following algebraic properties:
\begin{align*}
	\Lambda D&=(	\mathfrak a_{-}D)\Lambda,&\ D\Lambda&=\Lambda 	\mathfrak a_{+}D, & D\Lambda^\intercal&=\Lambda^\intercal 	\mathfrak a_{-}D, & \Lambda^\intercal D&=	\mathfrak a_{+}D\Lambda^\intercal. \end{align*}

Multiple orthogonal polynomials of type II and I on the step-line are given as entries of certain vector of polynomials.  Type II multiple orthogonal polynomial $B_n$ are the entries  of the  vector $B$, while type I  multiple orthogonal polynomials $A_n^{(a)}$, $a\in\{1,2\}$,  are the entries of the vectors  $A^{(a)}$ where
\begin{align*}
	B&\coloneq SX, & A^{(a)}&\coloneq H^{-1}\tilde{S}X^{(a)},  & a&\in\{1,2\}.
\end{align*}
The degrees of these polynomials are easily seen to be 
\begin{align*}%\label{eq:degrees}
	\deg B_{n} &=n , 	&\deg A^{(a)}_{n} &=\left\lceil\frac{n+2-a}{2} \right\rceil-1, & a&\in\{1,2\},
\end{align*}
here $\lceil x\rceil$ is the the smallest integer that is greater than or equal to $x$.
%The linear form is $ Q\coloneq H^{-1}\tilde{S}\varepsilon=A^{(1)}w^{(1)}+A^{(2)}w^{(2)}$.
	The following multiple orthogonality relations are satisfied
\begin{align*}
	\sum_{a=1}^2 \sum_{k=0}^\infty A^{(a)}_{n} (k) w^{(a)} (k)k^m &=0, &k &\in\left\{ 0 ,\dots, \deg B_{n-1}\right\}, \\
	\sum_{k=0}^\infty B_{n} (k) w^{(a)}(k) k^m &=0, &k &\in\left\{0 ,\dots, \deg A^{(a)}_{n-1}\right\},& a&\in\{1,2\}.
\end{align*}
According to \cite[Proposition 7]{afm} the type II polynomials can be expressed as the following quotient             of determinants
\begin{align}\label{eq:determinant_B}
	B_n&=\frac{1}{\tau_n}\begin{vNiceArray}{cw{c}{1cm}c|c}[margin] \Block{3-3}<\Large>{\mathscr M^{[n]}} & & &1 \\
		 & & & \Vdots \\
		&&&x^{n-1}\\
		\hline
		\mathscr{M}_{n,0} & \Cdots& \mathscr M_{n,n-1} & x^n
	\end{vNiceArray}, &n\in \N_0.
\end{align}
Hence, for the coefficients $p^j_n$ of the type II multiple orthogonal polynomials
\begin{align*}
	B_n=x^n+p^1_nx^{n-1}+\dots+p^{n-1}_n x+p_n^n
\end{align*}
we get determinantal expressions. For that aim, let us consider the  \emph{associated $\tau$-functions} $\tau^j_n$, $j\in\{1,\dots,n-1\}$, as the determinants of the matrix $\mathscr M^{[n,j]}$, $j\in\{1,\dots,n-1\}$, obtained from $\mathscr M^{[n]}$ by removing the $(n-j)$-th row $\begin{bNiceMatrix}
	\mathscr M_{n-j,0}&\Cdots & \mathscr M_{n-j,n-1}
\end{bNiceMatrix}$ and adding  as  last row the row $\begin{bNiceMatrix}
\mathscr M_{n,0}&\Cdots & \mathscr M_{n,n-1}
\end{bNiceMatrix}$. Then, by expanding the determinant in the RHS of \eqref{eq:determinant_B} by the last column we get
\begin{align}
	p^j_n=(-1)^j\frac{\tau^j_n}{\tau_n}.
\end{align}

Similarly \cite{afm} we have
\begin{align}\label{eq:determinant_A}
	A^{(a)}_n&=
	\frac{1}{\tau_{n+1}}
	{\NiceMatrixOptions{cell-space-limits = 2pt}\begin{vNiceArray}{cw{c}{1cm}c|c}[margin] \Block{3-3}<\Large>{\mathscr M^{[n]}} & & &	\mathscr{M}_{0,n} \\
			& & & \Vdots \\
			&&& \mathscr M_{n-1,n}\\
			\hline
			X^{(a)}_0 & \Cdots&X^{(a)}_{n-1}& X^{(a)}_n
		\end{vNiceArray}}, &n\in \N_0.
\end{align}

The Hessenberg matrix  $T\coloneq S\Lambda S^{-1}$ and its transposed matrix is $T^{\intercal}%=\tilde{J}^2=J_1+J_2
=H^{-1}\tilde{S}\Lambda^2 \tilde{S}^{-1}H$ give the 4-term recursion relations related to these multiple orthogonal polynomials.
Indeed, the following recurrence relations hold
\begin{align}\label{eq:recursion}
	TB&=zB, &T^\intercal A^{(a)}&=zA^{(a)}, &a&\in\{1,2\}.
\end{align}
Notice that  $T$ is a  tetradiagonal Hessenberg matrix, i.e.,
\begin{align}\label{eq:T_rec}
	T=(\Lambda^\intercal)^2\gamma+\Lambda^\intercal \beta+\alpha+\Lambda,
\end{align}
where $\alpha$, $\beta$, and $\gamma$ are the following diagonal matrices 
\begin{align*}
	\alpha&=\diag{(\alpha_0,\alpha_1,\dots)},& \beta&=\diag{(\beta_1,\beta_2,\dots)},&\gamma&=\diag{(\gamma_2,\gamma_3,\dots)}.
\end{align*}

The diagonal matrices $\alpha$, $\beta$, and $\gamma$ are related with $S$, $\tilde{S}$ and $H$ matrices by the following equations:
\begin{align}\label{eq:alphabetagamma}
\left\{
\begin{aligned}
		\alpha&=S^{[-1]}+\mathfrak a_{+}S^{[1]}=\mathfrak a_{+}S^{[1]}-S^{[1]}=\mathfrak a_{+}^2\tilde{S}^{[2]}-\tilde{S}^{[2]}+\tilde{S}^{[1]}(\mathfrak a_{-}\tilde{S}^{[1]}-\mathfrak a_{+}\tilde{S}^{[1]}),\\\beta&=H^{-1}\mathfrak a_{-}H(\mathfrak a_{+}\tilde{S}^{[1]}+\mathfrak a_{-}\tilde{S}^{[-1]})=H^{-1}\mathfrak a_{-}H(\mathfrak a_{+}\tilde{S}^{[1]}-\mathfrak a_{-}\tilde{S}^{[1]})=\mathfrak a_{+}S^{[2]}-S^{[2]}-S^{[1]}\mathfrak a_{-}\alpha,\\\gamma&=H^{-1}\mathfrak a^2_{-}H.
\end{aligned}\right.
\end{align}

We introduce the  following matrices $T^{(1)}$ and $T^{(2)}$ 
\begin{align*}T^{(a)}&\coloneq H^{-1}\tilde{S}\Lambda^{(a)}\tilde{S}^{-1}H=T^\intercal M^{(a)},&
 	M^{(a)}&\coloneq H^{-1}\tilde{S}I^{(a)}\tilde{S}^{-1}H,&a&\in\{1,2\}.
 \end{align*}
These matrices satisfy 
\begin{align*}
	T^{(a)} A^{(b)}&=\delta_{a,b} zA^{(b)}, & a,b\in\{1,2\}.
	%	T^{(1)}A^{(1)}&=zA^{(1)}, &T^{(2)}A^{(2)}&=zA^{(2)}, & T^{(1)}A^{(2)}&=0, & T^{(2)}A^{(1)}&=0.
\end{align*}
as well as $T^\intercal=T^{(1)}+T^{(2)}$. Hence, $T^{(a)}$ are matrices with all its superdiagonals but for the first and the second have zero entries.
%Moreover, we notice  that  $\hat T\coloneq H^{-1}\tilde{S}\Lambda\tilde{S}^{-1}H$ is such that $T^\top=\hat T^2$ and
%\begin{align*}
%	 \hat{T}A^{(1)}&=zA^{(2)},& \hat{T} A^{(2)}&=A^{(1)}.
%\end{align*}

%\subsection{Pascal matrices}
The Pascal matrices $L=(L_{n,m})$ have entries given by 
\begin{align*} 
L_{n,m}&=\begin{cases} \binom{n}{m},& n \geq m, \\ 0,& n<m,  \end{cases}
 \end{align*}
while  its inverse $L^{-1}=(\tilde L_{n,m})$ have entries 
\begin{align*}
\tilde L_{n,m}=\begin{cases} {}(-1)^{n+m} \binom{n}{m},& n \geq m, \\ 0&n<m.\end{cases}
\end{align*}
% it is:
%
%\begin{equation*}B=\begin{pmatrix}1&0&0&0&\dots\\1&1&0&0&\dots\\1&2&1&0&\dots\\\vdots&\ddots&\ddots&\ddots&\ddots\end{pmatrix}\end{equation*}
These matrices act on the monomial vectors as follows
\begin{align*}
	X(z+1)&=LX(z),&X(z-1)&=L^{-1}X(z).
\end{align*}
and can be  expanded as 
\begin{align*}
	L^{\pm1}&=I\pm\Lambda^{\intercal}D+(\Lambda^\intercal)^2D^{[2]}+\cdots
\end{align*}
with $D\coloneq\diag(1,2,3,\dots)$ and $D^{[n]}\coloneq \frac{1}{n}\diag(n^{(n)},(n+1)^{(n)},\dots)$
where $x^{(n)}=x(x-1)\cdots (x-n+1)$.
In terms of   Pascal matrices we  define the dressed Pascal  matrices 
\begin{align*}
	\Pi&\coloneq SLS^{-1},& \Pi^{-1}&\coloneq SL^{-1}S^{-1}.
\end{align*}
For type II  multiple orthogonal polynomials we find the following properties:
\begin{align*}
	B(z+1)&=\Pi B(z), &B(z-1)&=\Pi^{-1}B(z).
\end{align*}

We can proceed similarly for the type I multiple  orthogonal polynomials. We introduce the matrices $L^{\pm(a)}$, $a\in\{1,2\}$, with nonzero entries 
\begin{align*}
L^{\pm(a)}_{2n+a-1,2m+a-1}&=(\pm 1)^{n+m}\binom{n}{m}, & a&\in\{1,2\}, & n,m&\in\N_0, & n&\geq m,
\end{align*}
and zero for any other couple sub-indexes.
We will also use $L^{(a)}=L^{+(a)}$.
These matrices satisfy
\begin{align*}
	L^{(a)}	L^{-(a)}&=I^{(a)}, & a&\in\{1,2\}.
\end{align*}
The partial Pascal matrices can be expressed as band matrices, but in this case the even or odd diagonals are null, i.e.,
\begin{align*}
	L^{\pm(a)}=I^{(a)}\pm(\Lambda^\intercal)^2D^{(a),[1]}+(\Lambda^\intercal)^4D^{(a),[2]}+\cdots,
\end{align*}
where we have 
\begin{align*}
	D^{(1),[n]}_1&=\frac{1}{n}\diag (n^{(n)},0,(n+1)^{(n)},0,\dots), &D^{(2),[n]}&=\frac{1}{n}\diag (0,n^{(n)},0(n+1)^{(n)},\dots).
\end{align*}
Analogously, dressed Pascal matrices are  defined by
\begin{align*}\Pi^{\pm(a)}\coloneq H^{-1}\tilde{S}L^{\pm(a)}\tilde{S}^{-1}H,
 \end{align*}
and the following relations are satisfied
\begin{align*}
	A^{(a)}(z\pm 1)&=\Pi^{\pm (a)}A^{(a)}(z). 
\end{align*}

%We also consider corresponding discrete Cauchy transforms or second kind functions:
%\begin{align*}
%	D^{(a)}(z)&\coloneq \sum_{k=0}^\infty\frac{B(k)}{z-k}w^{(a)}(k), & a&\in\{1,2\}, &
%C(z)&\coloneq\sum_{k=0}^\infty\frac{w^{(1)}(k)A^{(1)}(k)+w^{(2)}(k)A^{(2)}(k)}{z-k}.
%\end{align*}
%That are, by definition, meromorphic functions with simple poles at $\N_0$.
%We  use the notation $X_*(x)\coloneq x^{-1}X(x^{-1})$ and $X^{(a)}_*(x)\coloneq x^{-1}X^{(a)}(x^{-1})$.
%Taking into account that 
%\begin{align*}
%	\frac{1}{z-k}&=\sum_{n=0}^\infty\frac{k^n}{z^{n+1}}=\big(X(k)\big)^\top X_*(z)=\big(X^{(a)}(k)\big)^\top X_*^{(a)}(z), & |z|&>k, & a&\in\{1,2\},
%\end{align*}
%we get the following asymptotic expressions
%\begin{align*}
%	D^{(a)}(z)&=H\tilde S^{-\top}X_*^{(a)}(z). & a&\in\{1,2\}, &
%C(z)&=S^{-\top}X_*(z),& z&\to\infty.
%\end{align*}
%Indeed, for $z\to\infty$ we have
%\begin{align*}
%	D^{(a)}(z)&=S\sum_{k=0}^\infty X(k)w^{(a)}(k)\big(X^{(a)}(k)\big)^\top X_*^{(a)}(z)=S\mathscr MX_*^{(a)}(z)=H\tilde S^{-\top}X_*^{(a)}(z), \\
%	C(z)&=H^{-1} \tilde S\sum_{k=0}^\infty (w^{(1)}(k)X^{(1)}(k)+w^{(2)}(k)X^{(2)}(k))\big(X(k)\big)^\top X_*(z)=\tilde S^{-\top}X_*(z).
%\end{align*}

\section{Discrete Multiple Orthogonal Polynomials and  Pearson equations}
\subsection{Pearson equation}
Let us consider that   weights $w^{(1)},w^{(2)}$ 
 subject to the following discrete  Pearson equations
\begin{align} \label{eq:Pearson1} 
	\theta(k+1) w^{(a)}(k+1)&=\sigma^{(a)}(k)w^{(a)}(k), & a&\in\{1,2\}.
	%\\  \label{eq:Pearson2}\theta(k+1) w^{(2)}(k+1)=\sigma^{(2)}(k)w^{(2)}(k),
\end{align}
where $\theta$, $\sigma^{(1)}$ and $\sigma^{(2)}$ are polynomials, and $\theta(0)=0$.
An important result regarding these type of weights is:
\begin{teo}
Let us assume that the weights $w^{(1)},w^{(2)}$ solve the Pearson equations \eqref{eq:Pearson1}. Then,  the corresponding moment matrix $\mathscr M$ satisfy the following matrix equation
\begin{align}
	 \label{eq:simhip}\theta(\Lambda) 
	 \mathscr M=L \mathscr M\left(L^{(1)}\sigma^{(1)}(\Lambda^{(1)})+L^{(2)}\sigma^{(2)}(\Lambda^{(2)})\right)^\intercal.
  \end{align}
\end{teo}
\begin{proof}
For the moment matrix is  $ \mathscr M$ we find
\begin{align*}\theta(\Lambda) \mathscr M&
=\sum^{\infty}_{k=1} \theta(k)X(k)\big(w^{(1)}X^{(1)}(k)+w^{(2)}X^{(2)}(k)\big)^\intercal\\
&=\sum^{\infty}_{k=0} \theta(k+1)X(k+1)\big(w^{(1)}(k+1)X^{(1)}(k+1)+w^{(2)}(k+1)X^{(2)}(k+1)\big)^\intercal)\\
&=\sum^{\infty}_{k=0} X(k+1)\big(\sigma^{(1)}(k)w^{(1)}(k)X^{(1)}(k+1)+\sigma^{(2)}(k)w^{(2)}(k)X^{(2)}(k+1)\big)^\intercal\\
&=\sum^{\infty}_{k=0} LX(k)\big(\sigma^{(1)}(k)w^{(1)}(k)X^{(1)}(k)^\intercal {L^{(1)}}^\intercal+\sigma^{(2)}(k)w^{(2)}(k)X^{(2)}(k)^\intercal  {L^{(2)}}^\intercal\big)\\
&=\sum^{\infty}_{k=0} LX(k)\big(X^{(1)}(k)^\intercal w^{(1)}(k)+X^{(2)}(k)^\intercal w^{(2)}(k)\big)\big(\sigma^{(1)}({\Lambda^{(1)}}^\intercal){L^{(1)}}^\intercal+\sigma^{(2)}({\Lambda^{(2)}}^\intercal){L^{(2)}}^\intercal\big)\\
&=L \mathscr M[L^{(1)}\sigma^{(1)}(\Lambda^{(1)})+L^{(2)}\sigma^{(2)}(\Lambda^{(2)})]^\intercal .
\end{align*}
\end{proof}

This matrix property implies for the tetradiagonal Hessenberg recursion matrix $T$ the following 
\begin{pro} 
Let us assume that the weights $w^{(1)},w^{(2)}$ solve the Pearson equations \eqref{eq:Pearson1},  for the recursion matrix $T$  in \eqref{eq:T_rec}, the following relation  holds
\begin{align}\label{eq:T}
	\Pi^{-1}\theta(T)=\sigma^{(1)}\big({T^{(1)}}^\intercal\big){\Pi^{(1)}}^\intercal+\sigma^{(2)}\big({T^{(2)}}^\intercal\big){\Pi^{(2)}}^\intercal.\end{align}
\end{pro}

\begin{proof} 
Using Equation  \eqref{eq:simhip} and the Gauss--Borel factorization we get
\begin{align*}
	\theta(\Lambda)S^{-1}H\tilde{S}^{-\intercal}=LS^{-1}H\tilde{S}^{-\intercal}\big(\sigma^{(1)}({\Lambda^{(1)}}^\intercal){L^{(1)}}^\intercal+\sigma^{(2)}({\Lambda^{(2)}}^\intercal){L^{(2)}}^\intercal\big), 
\end{align*}
grouping $S$ on one side and $\tilde{S} $ on the other side of the equation:
\begin{align*}
	S\theta(\Lambda)S^{-1}=SLS^{-1}H\tilde{S}^\intercal(\sigma^{(1)}({\Lambda^{(1)}}^\intercal){L^{(1)}}^\intercal+
	\sigma^{(2)}(	{\Lambda^{(2)}}^\intercal){L^{(2)}}^\intercal)\tilde{S}^{\intercal}H^{-1},
\end{align*} 
it can be written as
\begin{align*} \Pi^{-1}\theta(T)=H\tilde{S}^{-\intercal}\Big(\sigma^{(1)}(T^\intercal_1)(\tilde{S}^\intercal H^{-1})(H\tilde{S}^{-\intercal}){L^{(1)}}^\intercal+\sigma^{(2)}(T^\intercal_2)(\tilde{S}^\intercal H^{-1})(H\tilde{S}^{-\intercal}){L^{(2)}}^\intercal\Big)\tilde{S}^\intercal H^{-1},
\end{align*}
and finally we get Equation \eqref{eq:T}.
\end{proof}

\subsection{Hypergeometric discrete multiple orthogonal polynomials}
Let us write  the polynomials $\theta$ and $\sigma^{(a)}$, $a\in\{1,2\}$, in  Equations \eqref{eq:Pearson1}  as follows 
\begin{align}\label{eq:Pearson_coeffcients}
	\theta(z)&=z(z+c_1-1)\cdots(z+c_N-1), & \sigma^{(a)}(z)&=\eta^{(a)}\big(z+b^{(a)}_1\big)\cdots\big(z+b^{(a)}_{M^{(a)}}\big),  & M^{(a)},N&\in\N_0,& a&\in\{1,2\}.%\\ %\sigma^{(2)}(z)&={\eta^{(2)}}(z+b_1)\dots(z+b_{M_2}). 
\end{align}
In the subsequent discussion we require of generalized hypergeometric series, see \cite{LibrodeKF,LibrodeHypergeom},
%\begin{align*}
%	_MF_N(a_1,\dots,a_M;c_1,\dots,c_N;\eta)\coloneq \sum^{\infty}_{k=0}\frac{(a_1)_k\dots(a_M)_k}{(c_1+1)_k\dots(c_N+1)_k}\frac{\eta^k}{k!}
%\end{align*}
\begin{align*}
%	\label{Hypergeom}
	\pFq{M}{N}{b_1,\dots,b_M}{c_1,\dots,c_N}{\eta}\coloneq\sum_{k=0}^{\infty}\dfrac{(b_1)_k\cdots(b_M)_k}{(c_1)_k\cdots(c_N)_k}\dfrac{\eta^k}{k!},
\end{align*}
where we use the Pochhammer symbol $(b)_k\coloneq b(b+1)\cdots(b+k-1)$ and $(b)_0=1$,

Pearson equations \eqref{eq:Pearson1} determine uniquely the weights $w^{(1)},w^{(2)}$, up to multiplicative constant, to be
\begin{align}\label{eq:the_weights}
w^{(a)}(k)&=\dfrac{\big(b^{(a)}_1\big)_k\cdots\Big(b^{(a)}_{M^{(a)}}\Big)_k}{(c_1)_k\cdots(c_N)_k}\dfrac{{(\eta^{(a)})}^k}{k!}, & a&\in\{1,2\}.
\end{align}
In what follows we will  use logarithmic derivatives
\begin{align*}
	\vartheta^{(a)}&\coloneq\eta^{(a)}\frac{\partial\quad\quad}{\partial\eta^{(a)}}, & a&\in\{1,2\},
\end{align*}
as well as
\begin{align*}
	\vartheta&\coloneq\vartheta^{(1)}+\vartheta^{(2)}, & \tilde \vartheta&\coloneq\vartheta^{(1)}-\vartheta^{(2)}.
\end{align*}We introduce
$  t^{(a)}\coloneq\log{\eta^{(a)}}$, $a\in\{1,2\}$ and
\begin{align*}
t&\coloneq \frac{t^{(1)}+t^{(2)}}{2}, & \tilde t&\coloneq \frac{t^{(1)}-t^{(2)}}{2},
\end{align*}
so that $\eta^{(1)}=\Exp{t+\tilde t}$ and $\eta^{(2)}=\Exp{t-\tilde t}$ 
\begin{lemma}\label{lem:theta-rho}
	For the differential operators $\vartheta$ and $\tilde\vartheta$ we find
	\begin{align*}
		\vartheta&=\frac{\partial}{\partial t},&\tilde{\vartheta}=\frac{\partial}{\partial \tilde{t}}.
	\end{align*}
\end{lemma}
\begin{proof}
It follows from
	\begin{align*}
		\frac{\partial}{\partial t}&=\frac{\partial}{\partial t^{(1)}}+\frac{\partial}{\partial t^{(2)}}=\vartheta, &
		\frac{\partial}{\partial \tilde{t}}&=\frac{\partial}{\partial t^{(1)}}-\frac{\partial}{\partial t^{(2)}}=\tilde{\vartheta}. 
	\end{align*}
\end{proof}

An open question is whether the system of weights $\{w^{(1)},w^{(2)}\}$ is  perfect. That is the case if we are dealing with AT systems, see \cite{Ismail}. 
Following \cite[Examples 2.1 and 2.2]{Arvesu} we can give two families satisfying this condition:
\begin{lemma}[Perfect \& AT systems]\label{lemma:ATSystems}
	The system of weights  $\{w^{(1)},w^{(2)}\}$ with:
	\begin{enumerate}
	\item  
	\begin{align*}
		w^{(a)}(k)&=
		\dfrac{\big(b_1\big)_k\cdots\big(b_M\big)_k}{(c_1)_k\cdots(c_N)_k}
			\dfrac{(\eta^{(a)})^k}{k!}, & a&\in\{1,2\},
	\end{align*}
where $\eta^{(a)},b_i,c_j>0$, and
	\item 
\begin{align*}
	w^{(a)}(k)&=\dfrac{\big(b_1\big)_k\cdots\big(b_{M-1}\big)_k(b^{(a)})_k}{(c_1)_k\cdots(c_N)_k}\dfrac{\eta^k}{k!}, & a&\in\{1,2\},
	\end{align*}
	with $\eta,b_i,c_j,b^{(a)}>0$ 
\end{enumerate}
 are AT systems and consequently a perfect systems.
\end{lemma}
\begin{proof}
\begin{enumerate}
	\item  This fit  in \cite[Examples 2.1]{Arvesu}.
	\item  This fit  in \cite[Examples 2.2]{Arvesu}.
\end{enumerate}
\end{proof}

\begin{pro}
In terms of generalized hypergeometric functions,  the moment matrix can be written as a block matrix
{\NiceMatrixOptions{cell-space-limits = 3pt}\begin{align}\label{eq:matmomhyp}
	\mathscr M=\begin{bNiceArray}{cc|cc|cc|cw{c}{1cm}}[margin]\rho^{(1)}_0& \rho^{(2)}_0 & \cellcolor{Gray!10}\rho^{(1)}_1 & \cellcolor{Gray!10}\rho^{(2)}_1& \cellcolor{Gray!20}\rho^{(1)}_2 & \cellcolor{Gray!20}\rho^{(2)}_2 &\phantom{t}&\Cdots[shorten-start=-10pt,shorten-end=-2pt]
		 \\  \hline\cellcolor{Gray!10}\rho^{(1)}_1&\cellcolor{Gray!10}\rho^{(2)}_1 & \cellcolor{Gray!20}\rho^{(1)}_2 & \cellcolor{Gray!20}\rho^{(2)}_2& \rho^{(1)}_3 & \rho^{(2)}_3 &\phantom{t}&\Cdots[shorten-start=-10pt,shorten-end=-2pt] 
		 \\\hline
		 \cellcolor{Gray!20}\rho^{(1)}_2& \cellcolor{Gray!20}\rho^{(2)}_2 & \rho^{(1)}_3 & \rho^{(2)}_3& \rho^{(1)}_4 & \rho^{(2)}_4 &\phantom{t}&\Cdots[shorten-start=-10pt,shorten-end=-2pt] 
		 \\\hline
		 \Vdots[shorten-start=10pt,shorten-end=7pt]  &\Vdots[shorten-start=10pt,shorten-end=7pt] & \Vdots[shorten-start=10pt,shorten-end=7pt] & \Vdots[shorten-start=10pt,shorten-end=7pt]  & \Vdots[shorten-start=10pt,shorten-end=7pt]  &\Vdots[shorten-start=10pt,shorten-end=7pt]  &  \end{bNiceArray},
\end{align}}
where  the moments are expressed in terms of generalized hypergeometric series as follows
\begin{align}	\label{eq:hyp1}
	\rho^{(a)}_n&=({\vartheta^{(a)}})^n\rho^{(a)}_0,  &\rho^{(a)}_0&\coloneq	\pFq{M_a}{N}{b^{(a)}_1,\dots,b^{(a)}_{M^{(a)}}}{c_1,\dots,c_N}{\eta^{(a)}},& a&\in\{1,2\}, & n&\in\N_0.
\end{align}
\end{pro}
\begin{proof} It follows from the definition of the moment matrix $\mathscr M$.
Equation \eqref{eq:hyp1} determining the moments in terms of two families of generalized hypergeometric series   follows from Equation \eqref{eq:matmomhyp} and \cite{Manas_Fernandez-Irisarri}.\end{proof}

Given a function $f=f(\eta)$ we consider the covector  $\delta_n(f)\coloneq\begin{bNiceMatrix}
	f &\vartheta f&\Cdots&\vartheta^{n-1} f
\end{bNiceMatrix}$.
For  two functions $f_1(\eta^{(1)})$  we consider the double Wronskian of the covectors $\delta_n(f_1)$ and $\delta_n(f_2)$ given by
{\NiceMatrixOptions{cell-space-limits = 1pt}\begin{align*}
	\mathscr W_{2n}[f_1,f_2]&\coloneq \begin{vNiceArray}{w{c}{30pt}w{c}{30pt}|w{c}{30pt}w{c}{30pt}|w{c}{30pt}w{c}{30pt}|ccc|w{c}{30pt}w{c}{30pt}}[margin]
		f _1 &f_2&\vartheta f_1& \vartheta f_2&\vartheta^2f_1&\vartheta^2f_2& \phantom{t}&\Cdots[shorten-start=-5pt,shorten-end=-5pt]&\phantom{t}&\vartheta^{n-1}f_1&\vartheta^{n-1}f_2\\\hline
		\vartheta f_1& 	\vartheta f_2&\vartheta^2f_1&\vartheta^2f_2&\vartheta^3f_1&\vartheta^3f_1&\phantom{t}&\Cdots[shorten-start=-5pt,shorten-end=-5pt]&\phantom{t}& \vartheta^{n}f_1& \vartheta^{n}f_2\\\hline
		\vartheta^2 f_1&\vartheta^2 f_2&\vartheta^3f_1&\vartheta^3f_2&\vartheta^4f_1&\vartheta^4f_2&\phantom{t}&\Cdots[shorten-start=-5pt,shorten-end=-5pt]&\phantom{t}&\vartheta^{n}f_1& \vartheta^{n}f_2\\\hline
		\Vdots& \Vdots& \Vdots&\Vdots&\Vdots&\Vdots&&&&\Vdots&\Vdots\\\\\\\hline
		\vartheta^{2n-1}f_1& 	\vartheta^{2n-1}f_2&\vartheta^{2n}f_1& \vartheta^{2n}f_2&\vartheta^{2n+1}f_1& \vartheta^{2n+1}f_2&\phantom{t}&\Cdots[shorten-start=-5pt,shorten-end=-5pt]&\phantom{t}&\vartheta^{3n-1}f_1&\vartheta^{3n-1}f_2
	\end{vNiceArray},
\end{align*}
the double Wronskian of the covectors $\delta_{n+1}(f_1)$ and $\delta_n(f_2)$ given by
	\begin{align*}
	\mathscr W_{2n+1}[f_1,f_2]&\coloneq 
	\begin{vNiceArray}{w{c}{30pt}w{c}{30pt}|w{c}{30pt}w{c}{30pt}|w{c}{30pt}w{c}{30pt}|ccc|w{c}{30pt}w{c}{30pt}|w{c}{30pt}}[margin]
		f _1 &f_2&\vartheta f_1& \vartheta f_2&\vartheta^2f_1&\vartheta^2f_2& \phantom{t}&\Cdots[shorten-start=-5pt,shorten-end=-5pt]&
		\phantom{t}%\phantom{t}
		&\vartheta^{n-1}f_1&\vartheta^{n-1}f_2&\vartheta^{n}f_1\\\hline
		\vartheta f_1& 	\vartheta f_2&\vartheta^2f_1&\vartheta^2f_2&\vartheta^3f_1&\vartheta^3f_1&\phantom{t}&\Cdots[shorten-start=-5pt,shorten-end=-5pt]&\phantom{t}& \vartheta^{n}f_1& \vartheta^{n}f_2&\vartheta^{n+1}f_1\\\hline
		\vartheta^2 f_1&\vartheta^2 f_2&\vartheta^3f_1&\vartheta^3f_2&\vartheta^4f_1&\vartheta^4f_2&\phantom{t}&\Cdots[shorten-start=-5pt,shorten-end=-5pt]&\phantom{t}&\vartheta^{n+1}f_1& \vartheta^{n+1}f_2&\vartheta^{n+2}f_1\\\hline
		\Vdots& \Vdots& \Vdots&\Vdots&\Vdots&\Vdots&&&&\Vdots&\Vdots&\Vdots\\\\\\\hline
		\vartheta^{2n-1}f_1& 	\vartheta^{2n-1}f_2&\vartheta^{2n}f_1& \vartheta^{2n}f_2&\vartheta^{2n+1}f_1& \vartheta^{2n+1}f_2&\phantom{t}&\Cdots[shorten-start=-5pt,shorten-end=-5pt]&\phantom{t}&\vartheta^{3n-1}f_1&\vartheta^{3n-1}f_2&\vartheta^{3n}f_1\\
		\hline
		\vartheta^{2n}f_1& 	\vartheta^{2n}f_2&\vartheta^{2n+1}f_1& \vartheta^{2n+1}f_2&\vartheta^{2n+2}f_1& \vartheta^{2n+2}f_2&\phantom{t}&\Cdots[shorten-start=-5pt,shorten-end=-5pt]&\phantom{t}&\vartheta^{3n}f_1&\vartheta^{3n}f_2&\vartheta^{3n+1}f_1
	\end{vNiceArray}.
\end{align*}}
We refer to these objects as double $\delta$-Wronskians.
Then, the $\tau$-function \eqref{eq:tau} can be written as the double $\delta$-Wronskian of two generalized hypergeometric series as follows
\begin{align*}
	\tau_n=\mathscr W_n\left[\pFq{M_1}{N}{b^{(1)}_1,\dots,b^{(1)}_{M^{(1)}}}{c_1,\dots,c_N}{\eta^{(1)}},\pFq{M_2}{N}{b^{(2)}_1,\dots,b^{(2)}_{M^{(2)}}}{c_1,\dots,c_N}{\eta^{(2)}}
	\right].
\end{align*}
Double or 2-component Wronskians were first considered, within the context of integrable systems, by  Freeman, Gilson and  Nimmo  in \cite{Nimmo}, see also \cite{Guil-Manas}.

%The associated $\tau$-function $\tilde \tau_n$  is the determinant of the matrix built form $\mathscr M_n$ by replacing the last row by the next row of $\mathscr M$.

\begin{lemma}\label{lemma:tildetau}
		For the associated $\tau$-function we find
	\begin{align*}
		 \tau^1_n=\vartheta \tau_n.
	\end{align*}
\end{lemma}
\begin{proof}
		It follows from the multi-linearity of the determinant and  Equation \eqref{eq:matmomhyp}  in where the even columns depend only on $\eta^{(1)}$ and the odd columns on ${\eta^{(2)}}$.
\end{proof}

\begin{pro}\label{pro:Hptau}
The following expressions in terms of the $\tau$ functions hold true
\begin{align}\label{eq:H_p_tau}
H_n&=\frac{\tau_{n+1}}{\tau_n}, &	p^1_n=-\vartheta \log\tau_n.
\end{align}
\end{pro}
	
\begin{proof}
		It can be proven \cite{afm} 
	\begin{align}\label{eq:H_p1n}
		H_n&=\frac{\tau_{n+1}}{\tau_n}, &p^1_n=-\frac{ \tau^1_n}{\tau_n}.
	\end{align}
Then, using the previous Lemma \ref{lemma:tildetau} we get the result.
\end{proof}

\begin{pro}\label{pro:M}
The moment matrix satisfies
\begin{align*}
	\vartheta^{(a)}\mathscr M&=\mathscr M{\Lambda^{(a)}}^\intercal,& a&\in\{1,2\} .
\end{align*}
\end{pro}
\begin{proof}It follows from  Equations \eqref{eq:hyp1}. \end{proof}
\begin{rem}\label{rem:extending_tau}
The previous results apply to more general situations than those provided by the Pearson equation and the generalized hypergeometric series. We only need to assume that $\vartheta w^{(a)}(k)=kw^{(a)}(k)$. Using the double $\delta$-Wronskian, we can construct a tau function as follows:
\begin{align*}
	\tau_n&=\mathscr W_n[\rho^{(1)}_0,\rho^{(2)}_0], & \rho^{(a)}_0&=\sum_{n=0}^\infty w^{(a)}(k), & a&\in\{1,2\}.
\end{align*}
Then, Lemma \ref{lemma:tildetau} and Propositions \ref{pro:Hptau} and \ref{pro:M} hold true.
\end{rem}

\subsection{Laguerre-Freud matrix}
\begin{teo}[Laguerre--Freud matrix]Let us assume that the weights fulfill discrte Pearson equation \eqref{eq:Pearson1},
	%and \eqref{eq:Pearson2},
	 where $\theta$, $\sigma^{(1)}$ and $\sigma^{(2)}$ are polynomials with $\deg{\theta}=N$ and $\max{(\deg{\sigma^{(1)}},\deg{\sigma^{(2)}})}=M$.Then, the Laguerre-Freud matrix  given by 
\begin{align}\label{eq:FL} \Psi=\Pi^{-1}\theta(T)=\sigma^{(1)}({T^{(1)}}^\intercal){\Pi^{(1)}}^\intercal+\sigma^{(2)}({T^{(2)}}^\intercal){\Pi^{(2)}}^\intercal
\end{align}
  is a banded matrix with lower bandwidth $2M$ and upper bandwidth $N$, as follows:
\begin{align*}\Psi=(\Lambda^\intercal)^{2M}\psi^{(-2M)}+\cdots+\psi^{(N)}\Lambda^N.
\end{align*}
Moreover,
the following connection formulas are fulfilled:
\begin{align}\label{eq:P}
	\theta(z)B(z-1)&=\Psi B(z), \\ \label{eq:A1} 
(A^{(a)}(z+1))^\intercal 	\sigma^{(a)}(z)	&=(A^{(a)}(z))^\intercal \Psi, & a&\in\{1,2\}.
	% \\ \label{eq:A2} \Psi^\intercal A^{(2)}(z)&=\sigma^{(2)}(z)A^{(2)}(z+1). 
\end{align}
\end{teo}
\begin{proof} If $\deg{\theta}=N$ then $\theta(T)=\Lambda^N+M^{(N-1)}\Lambda^{N-1}+\cdots$ so it will have $N$ possible nonzero superdiagonals and using the same argument we conclude that it has  $2M$ subdiagonals possibly nonzero.
	
	For  the type II polynomials $B(z)$ we have
\begin{align*}
	\Psi B(z)=\Pi^{-1} \theta(T)B(z)=\Pi^{-1}\theta(z)B(z)=\theta(z)B(z-1)
\end{align*}
and for the type I polynomials $A^{(a)}$ we find
\begin{align*}
	\Psi^\intercal A^{(a)}(z)=\left(\Pi^{(1)}\sigma^{(1)}(T^{(1)})+\Pi^{(2)}\sigma^{(2)}(T^{(2)})\right)A^{(a)}(z)=\Pi^{(a)}\sigma^{(a)}(z)A^{(a)}(z)=\sigma^{(a)}(z)A^{(a)}(z+1).
\end{align*}
\end{proof}
%For the Laguerre-Freud structure matrix we have a compatibility equation which we will refer as \textit{Compatibility I}:

\begin{lemma}\label{lemma:thecnical_lemma_X}
	If the matrix $M$ is such that $MX(z)=0$ then $M=0$.
\end{lemma}
\begin{proof}
	From $MX(z)=0$ we conclude that $M \frac{1}{n!}X^{(n)}(0)=0$ but $\frac{1}{n!}X^{(n)}(0)=e_n$, where $e_n$ is the vector with all its entries zero but for a unit at the $n$-th entry. Therefore, we conclude that the $n$-th column of $M$ has all its entries equal to zero. Consequently, $M=0$
\end{proof}

\begin{pro}[Compatibility I] For recursion  and Freud--Laguerre matrices the following compatibility relation is fulfilled
\begin{align}\label{eq:comp1Psi} [\Psi,T]&=\Psi .
%	\\  \label{eq:comp1PsiT}
%[T^{(1)}+T^{(2)},\Psi^\intercal]&=\Psi^\intercal.
\end{align}
\end{pro}
\begin{proof}
%	Relation \eqref{eq:comp1Psi} can be proven from recurrence relation. Indeed, 
%we can write $TB(z-1)=(z-1)B(z-1)$, if we multiply the last equation by $\theta(z)$ and using \eqref{eq:P} we get
%\begin{align*}
%	T\theta(z)B(z-1)=(z-1)\theta(z)B(z-1).
%\end{align*}

Using \eqref{eq:P} and \eqref{eq:recursion} we find
\begin{align*} 
	T\Psi B(z)&=T\theta(z)B(z-1)=\theta(z)(z-1)\theta(z)B(z-1), &
	\Psi(T-I)B(z)&=\Psi (z-1)B(z)=(z-1)\theta(z) B(z-1).
\end{align*}
Therefore, $([\Psi,T]-\Psi)B(z)=0$.  Consequently, we get $([\Psi,T]-\Psi)S X(z)=0$ and according to Lemma \ref{lemma:thecnical_lemma_X} we deduce that $([\Psi,T]-\Psi)S =0$, and as $S$ is an unitriangular matrix, with inverses, we deduce that $[\Psi,T]-\Psi$ is the zero matrix.

%From here we deduce that 
%\begin{align*} 
%	T\Psi B(z)=(z-1)\Psi B(z)=\Psi(T-I)B(z),
%\end{align*}
%so that
%\begin{align*} T\Psi=\Psi T-\Psi
%\end{align*}
%and Equation \eqref{eq:comp1Psi}  follows.

%The relation \eqref{eq:comp1PsiT} can be proven transposing \eqref{eq:comp1Psi}:
%\begin{align*}  [\Psi,T]^\intercal=\Psi^\intercal
%\end{align*} 
%so that
%\begin{align*} 
%	\Psi^\intercal=(\Psi T)^\intercal-(T\Psi)^\intercal=T^\intercal \Psi^\intercal-\Psi^\intercal T^\intercal 
%\end{align*}
%and we get
%\begin{align*} 
%	\Psi^\intercal=[T^\intercal,\Psi^\intercal]=\left[T^{(1)}+T^{(2)},\Psi^\intercal\right]. 
%\end{align*}
\end{proof}

\subsection{Contiguous hypergeometric relations}
%\subsubsection{}
For generalized  hypergeometric series we have the following  contiguous  relations
\begin{align*}
	%(\vartheta+a_i)_MF_N(a_1,\dots,a_i,\dots,a_M;b_1,\dots,b_N;\eta)&=a_i {}_MF_N(a_1,\dots,a_i+1,\dots,a_M;b_1,\dots,a_N;\eta), \\
	\left(\eta\frac{\d\quad}{\d\eta}+b_i\right)\,	\pFq{M}{N}{b_1,\dots,b_M}{c_1,\dots,c_N}{\eta}&=b_i\,\pFq{M}{N}{b_1,\dots,b_i+1,\dots,b_M}{c_1,\dots,c_N}{\eta}, & i&\in\{1,\dots,M\}
	\\
\left(\eta\frac{\d\quad}{\d\eta}+c_j-1\right) \,\pFq{M}{N}{b_1,\dots,b_M}{c_1,\dots,c_N}{\eta}&=(c_j-1)\, \pFq{M}{N}{b_1,\dots,b_M}{c_1,\dots,c_j-1,\dots,c_N}{\eta},  & j&\in\{1,\dots,N\}
%\\\frac{\d}{\d\eta} \pFq{M}{N}{b_1,\dots,b_M}{c_1,\dots,c_N}{\eta}&=\kappa \,\pFq{M}{N}{b_1+1,\dots,b_M+1}{c_1+1,\dots,c_N+1}{\eta} 
\end{align*}
%where $\kappa:=\frac{\prod^M_{i=1}b_i}{\prod^N_{j=1}c_j}.$

Let us denote by $	{}_i\Theta^{(a)}$ the shifting the $b^{(a)}_i\to b^{(a)}_i+1$ and by $\Theta_j$ the shifting $c_j\to c_j+1$.
These contiguity relations translate into the moment matrix as follows: 
\begin{teo}[Contiguous relations for the moment matrix]
The following equations for the moment matrix hold:
\begin{align}\label{rel:hyp1a}
	\mathscr M{\Lambda^{(a)}}^\intercal+b^{(a)}_i\mathscr M&=b^{(a)}_i {}_i\Theta^{(a)}\mathscr M, & i&\in\{1,\dots,M^{(a)}\}, & a&\in\{1,2\},\\ 
		\label{rel:hyp1b} \mathscr M{\Lambda^2}^{\intercal}+(c_j-1)\mathscr M&=(c_j-1)\Theta_j\mathscr M ,&j&\in\{1,\dots,N\}
%\\ \label{rel:hyp2b} \mathscr M\Lambda^{(2)}^\intercal+(d_i-1)\mathscr M&=(d_i-1)T_i^{(2)}\mathscr M ,\\ 
%	\label{rel:hyp1c}
%	\frac{\d}{\d {\eta^{(1)}}}\mathscr M&=\kappa^{(1)}\boldsymbol\Theta^{(1)}\mathscr M-\mathscr Mp_2, \\ \label{rel:hyp2c} \frac{\d}{\d {\eta^{(2)}}}\mathscr M&=\kappa^{(2)}\boldsymbol\Theta^{(2)}\mathscr M-\mathscr MI^{(1)}. 
\end{align}
\end{teo}
\begin{proof}To prove Equation \eqref{rel:hyp1a} we notice that,
according to \eqref{eq:matmomhyp},  we can write the following equations for moment matrix entries 
\begin{align*}
    \big(\vartheta^{(1)}+b_i^{(1)}\big)\mathscr M_{n,2m}&=b_i^{(1)}\Theta_i^{(1)}\mathscr M_{n,2m}, \\
\big(\vartheta^{(1)}+b_i^{(1)}\big)\mathscr M_{n,2m+1}&=b_i^{(1)} \mathscr M_{n,2m+1}=\Theta_i^{(1)}\mathscr M_{n,2m+1},\\
  \big(\vartheta^{(2)}+b_i^{(2)}\big)\mathscr M_{n,2m}&=b_i ^{(2)}\mathscr M_{n,2m}=b_i^{(2)}\Theta_i^{(2)}\mathscr M_{n,2m+1}, \\ \big(\vartheta^{(2)}+b_i^{(2}\big)\mathscr M_{n,2m+1}&=b_i^{(2)}\Theta_i^{(2)}\mathscr M_{n,2m+1}.
\end{align*}
with $n,m\in\N_0$. Let us  prove Equation \eqref{rel:hyp1b} by noticing that the entries of the moment matrix satisfy
\begin{align*}
	(\vartheta+c_j-1)\mathscr M_{(n,2m)}&=(c_j-1)\Theta_i\mathscr M_{(n,2m)}, & 	(\vartheta+c_j-1)\mathscr M_{(n,2m+1)}&=(c_j-1)\Theta_i\mathscr M_{(n,2m+1)}.
\end{align*} 
so that
\begin{align*}
	\mathscr M({\Lambda^{(1)}}^\intercal+{\Lambda^{(2)}}^\intercal)+(c_j-1)\mathscr M&=(c_j-1)\Theta_j\mathscr M.
\end{align*}
%
%To prove \eqref{rel:hyp1c}, we can write the following relations for moment matrix elements:
%\begin{align*}
%	\frac{\d}{\d {\eta^{(1)}}}\mathscr M_{(n,2m)}&=\kappa^{(1)}\boldsymbol\Theta^{(1)}\mathscr M_{(n,2m)}, &\frac{\d}{\d {\eta^{(1)}}}\mathscr M_{(n,2m+1)}&=0, \\
%		\frac{\d}{\d {\eta^{(2)}}}\mathscr M_{(n,2m)}&=0, &\frac{\d}{\d {\eta^{(1)}}}\mathscr M_{(n,2m+1)}&=\kappa^{(2)}\boldsymbol\Theta^{(2)}\mathscr M_{(n,2m+1)}, 
%\end{align*}
%so we can see that $\frac{d}{d {\eta^{(1)}}}\mathscr M=\kappa^{(1)}\boldsymbol\Theta^{(1)}\mathscr M-\mathscr Mp_2$ is true.
\end{proof}
\subsection{Connection matrices}
\begin{defi} Connection matrices are defined as
\begin{align}
	\label{def:io}
	 _i\omega^{(a)}&\coloneq(_i\Theta^{(a)}H)^{-1}( _i\Theta^{(a)}\tilde{S})(\Lambda^{(a)}+b_i^{(a)})\tilde{S}^{-1}H, &i&\in\{1,\dots,M^{(a)}\}, &a&\in\{1,2\}, \\
	 \label{def:iO} 
	 _i\Omega^{(a)}&\coloneq S( _i\Theta^{(a)}S)^{-1}, &i&\in\{1,\dots,M^{(a)}\}, &a&\in\{1,2\}, \\\
	  \label{def:oi} 
	  \omega_j&\coloneq(\Theta_jH)^{-1}(\Theta_j\tilde{S})(\Lambda^2+c_j-1)\tilde{S}^{-1}H,&j&\in\{1,\dots,N\} ,\\
	  \label{def:Oi} 
	  \Omega_j&\coloneq S(\Theta_jS)^{-1},&j&\in\{1,\dots,N\}.%
%	\\ \label{def:o} \boldsymbol\omega^{(1)}&=(\boldsymbol\Theta^{(1)}H)^{-1}(\theta^{(1)}\tilde{S})B_1^{-1}\Lambda^{(1)}\tilde{S}^{-1}H, \\ \label{def:O} \boldsymbol\Omega^{(1)}&=SB(\boldsymbol\Theta^{(1)}S)^{-1}, 
\end{align}
%and connection matrices ${}_i\omega^{(2)}$, ${}_i\Omega^{(2)}$, $\omega^{(2)}$ and $\Omega^{(2)}$ are defined analogously.
\end{defi}

\begin{lemma}
The following relations between connection matrices are fulfilled
\begin{align} \label{relconmat:io} _i{\omega^{(a)}}^\intercal&=b_i^{(a)} {}_i\Omega^{(a)}, &i&\in\{1,\dots,M^{(a)}\}, &a&\in\{1,2\}, \\ \label{relconmat:oi}{\omega_j}^\intercal&=(c_j-1)\Omega_j, &j&\in\{1,\dots,N\}.
	%\\ \label{relconmat:o} %\kappa^{(n)}\boldsymbol\Omega^{(a)}&={\boldsymbol\omega^{(a)}}^\intercal.
 \end{align}
%where $a \in\{1,2\}$.
\end{lemma}
\begin{proof}
Let us prove Equation \eqref{relconmat:io}. From  the definition \eqref{rel:hyp1a} and the Gauss--Borel factorization \eqref{eq:LU} of the moment matrix $\mathscr M$ we find
\begin{align*}
	S^{-1}H\tilde{S}^{-\intercal}{\Lambda^{(a)}}^\intercal+b_i^{(a)}S^{-1}H\tilde{S}^{-\intercal}&=
	b_i^{(a)}({}_i\Theta^{(a)}S)^{-1}({}_i\Theta^{(a)}H)({}_i\Theta^{(a)}\tilde{S})^{-\intercal}, &i&\in\{1,\dots,M^{(a)}\}, &a&\in\{1,2\},
	\end{align*} 
that after some cleaning leads to 
\begin{align*}
	H\tilde{S}^{-\intercal}({\Lambda^{(a)}}^\intercal+b_i^{(a)})({}_i\Theta^{(a)}\tilde{S})^\intercal({}_i\Theta^{(a)}H)^{-1}&=b_i^{(a)}S({}_i\Theta^{(a)}S)^{-1},&i&\in\{1,\dots,M^{(a)}\}, &a&\in\{1,2\}.
\end{align*}
Analogously, Equation \eqref{relconmat:oi} can be proved using \eqref{rel:hyp1b} and \eqref{eq:LU}.
\end{proof}

\begin{pro}%(Banded structure of connection matrices)
Connection matrices have a banded triangular structure with three non-null diagonals. 
In the one hand, for $i\in\{1,\dots,M^{(a)}\}$, $a\in\{1,2\}$, we have
\begin{align}\label{est:io_0}
	{}_i\omega^{(a)}&=b^{(a)}_iI+\big({}_i\omega^{(a)}\big)^{[1]}\Lambda+\big({}_i\omega^{(a)}\big)^{[2]}\Lambda^2, \\ \label{est:iO_0} {}_i\Omega^{(a)}&={{}\Lambda^\intercal}^2\frac{1}{b^{(a)}_i}\big({}_i\omega^{(a)}\big)^{[2]}+\Lambda^\intercal 
	\frac{1}{b^{(a)}_i} \big({}_i\omega^{(a)}\big)^{[1]}+I, 
\end{align}
with
\begin{align*}
\big({}_i\omega^{(a)}\big)^{[1]}&\coloneq b_i^{(a)}(S^{[1]}-{}_i\Theta^{(a)}S^{[1]})=({}_i\Theta^{(a)}H)^{-1}(\mathfrak a_{-}H)\big(\mathfrak a_{+}({}_i\Theta^{(a)}\tilde{S}^{[1]})I^{(\pi(a))}-I^{(a)}(\mathfrak a_{-}\tilde{S}^{[1]})\big),\\
\big({}_i\omega^{(a)}\big)^{[2]}&\coloneq({}_i\Theta^{(a)}H)^{-1}I^{(a)}\mathfrak a_{-}^2H,
\end{align*}
where $\pi\in S_2$ is the permutation $\pi(1)=2$ and $\pi(2)=1$.
%Moreover,
%\begin{align*}
%	\frac{1}{b^{(1)}_i} \big({}_i\omega^{(1)}\big)^{[1]}&=({}_i\Theta^{(1)}H)^{-1}(\mathfrak a_{-}H)\big(\mathfrak a_{+}({}_i\Theta^{(1)}\tilde{S}^{[1]})p_2-I^{(1)}(\mathfrak a_{-}\tilde{S}^{[1]})\big), \\
%	\frac{1}{b^{(2)}_i} \big({}_i\omega^{(2)}\big)^{[1]}&=({}_i\Theta^{(2)}H)^{-1}(\mathfrak a_{-}H)\big(\mathfrak a_{+}({}_i\Theta^{(2)}\tilde{S}^{[1]})I^{(1)}-p_2(\mathfrak a_{-}\tilde{S}^{[1]})\big).
%\end{align*}
In the other hand, for $j\in\{1,\dots,N\}$, we find
\begin{align}	
	\omega_j &=(c_j-1)I+	\omega_j^{[1]}\Lambda+	\omega_j^{[2]}\Lambda^2,\\
	\Omega_j&=I+\Lambda^\intercal \frac{1}{c_j-1}\omega_j^{[1]}+(\Lambda^\intercal)^2 \frac{1}{c_j-1} \omega_j^{[2]},
\end{align}
with
\begin{align*}
\omega_j^{[1]}&\coloneq (c_j-1)(S^{[1]}-\Theta_jS^{[1]})=\mathfrak a_{-}H(\Theta_jH)^{-1}\big(\mathfrak  a_{+}(\Theta_j\tilde{S}^{[1]})-\mathfrak a_{-}\tilde{S}^{[1]}\big) ,\\
\omega_j^{[2]}&\coloneq (\Theta_jH)^{-1}\mathfrak a_{-}^2H.
\end{align*}

\end{pro}
\begin{proof}
Above relations  are found using  equations \eqref{relconmat:io},  \eqref{def:io} and \eqref{def:iO}. 
Let us show only the cases  ${}_i\omega^{(1)}$ and ${}_i\Omega^{(1)}$, as  the others cases are proven similarly.
Departing from Equations \eqref{def:io} and \eqref{def:iO} we see that 
\begin{align*}
	{}_i\omega^{(a)}&= ({}_i\Theta^{(a)}H)^{-1}{}_i\Theta^{(a)}(I+\Lambda^\intercal\tilde{S}^{[1]}+(\Lambda^\intercal)^2\tilde{S}^{[2]}+\cdots)(I^{(a)}\Lambda^2+b_i^{(a)})(I+\Lambda^\intercal\tilde{S}^{[-1]}+(\Lambda^\intercal)^2\tilde{S}^{[-2]}+\cdots)H ,\\
{}_i\Omega^{(a)}&=(I+\Lambda^\intercal S^{[1]}+\cdots){}_i\Theta^{(a)}(I+\Lambda^\intercal S^{[-1]}+\cdots),
\end{align*}
 and using \eqref{relconmat:io} we get the result.

\end{proof}
\begin{teo} Vectors of polynomials $A^{(a)}$, $a\in\{1,2\}$, and $B$ fulfill the following connection formulas 
\begin{align}\label{con:io1} 
	_i\omega^{(a)}A^{(a)}(z)&=(z+b_i^{(a)})\big(_i\Theta^{(a)}A^{(a)}(z)\big), \\\label{con:iO1} _i
	\Omega^{(1)} {}_i\Theta^{(a)}B(z)&=B(z), \\ \label{con:oi1}
	 \omega_jA^{(a)}(z)&=(z+c_i-1)\big(\Theta_i A^{(a)}(z)\big), \\ \label{con:Oi1}
	  \Omega_j\Theta_jB(z)&=B(z), 
  \end{align}
with $i\in\{1,\dots,M^{(a)}\}$, $a\in\{1,2\}$,  and $j\in\{1,\dots,N\}$.
\end{teo}
\begin{proof} Equations \eqref{con:io1}, \eqref{con:iO1} , \eqref{con:oi1} and \eqref{con:Oi1}
can be proved through the action of connection matrix on vectors $A^{(a)}$ and $B$. 
Let us check \eqref{con:io1}:
\begin{align*}
	{}_i\omega^{(a)}A^{(a)}(z)=({}_i\Theta^{(a)}H)^{-1}({}_i\Theta^{(a)}\tilde{S})(\Lambda^{(a)}+b_i^{(a)})X^{(a)}(z)=({}_i\Theta^{(a)}H)^{-1}({}_i\Theta^{(a)}\tilde{S})\big(z+b_i^{(a)}\big)X^{(a)}(z),
\end{align*}
and immediately \eqref{con:io1} is proved.
\end{proof}

\section{Multiple Toda systems and hypergeometric $\tau$-functions}

 \subsection{Continuous case}
 
 Let us now  explore how the hypergeometric discrete multiple orthogonal polynomials leads to solutions, in terms of $\tau$-functions, which are double Wronskians of generalized hypergeometric series, to multiple Toda type systems. 
 
Let us start by introducing the strictly lower triangular matrices
\begin{align*}
	\phi^{(a)}&\coloneq (\vartheta^{(a)}S)S^{-1}, &  	\tilde{\phi}^{(a)}&\coloneq(\vartheta^{(a)}\tilde{S})\tilde{S}^{-1}, & a&\in\{1,2\}.
\end{align*}
For them we have the following two lemmas
\begin{lemma}
	The following relations are fulfilled
	\begin{align}\label{derP:1}
		\vartheta^{(a)}B&=\phi^{(a)}B, & a&\in\{1,2\}.
	\end{align}
\end{lemma}
\begin{proof}
	From $B=SX$ we get
	\begin{align*} 
		\vartheta^{(a)} B=(\vartheta^{(a)}S)X=(\vartheta^{(a)}S)S^{-1}SX=(\vartheta^{(a)}S)S^{-1}B.  \end{align*}
\end{proof}
\begin{lemma}
	The following relations are fulfilled
	\begin{align}
		\vartheta^{(a)}(HA^{(a)})&=\tilde{\phi}^{(a)}HA^{(a)}, & a&\in\{1,2\}.
	\end{align}
\end{lemma}
\begin{proof}
	We have that
	\begin{align*}
		\vartheta^{(a)}(HA^{(a)})=\vartheta(\tilde{S})X^{(a)}=\tilde{\phi}HA^{(a)}.
	\end{align*}
\end{proof}
\begin{pro}\label{pro:alphabetagamma}
	The following equations are fulfilled
	\begin{align}\label{satow1}
		\vartheta^{(a)}H-\phi^{(a)}H-
		H{{}\tilde\phi^{(a)}}^\intercal&={T^{(a)}}^\intercal H, & a&\in\{1,2\}.
		%\label{satow2}(\vartheta^{(2)})_kH-\phi^{(2)}_kH-H(\tilde{\phi}^{(2)}_k)^\intercal=(J_1^\intercal)^kH ,
	\end{align}
\end{pro}
\begin{proof}Recall that $\mathscr M=S^{-1}H\tilde{S}^{-\intercal}$ so that
	\begin{align*}\vartheta^{(a)}(S^{-1}H\tilde{S}^{-\intercal})=-S^{-1}(\vartheta^{(a)}S)S^{-1}H\tilde{S}^{-\intercal}
		+S^{-1}(\vartheta^{(a)}H)\tilde{S}^{-\intercal}-S^{-1}H\tilde{S}^{-\intercal}(\vartheta^{(a)}\tilde{S})^{\intercal}\tilde{S}^{-\intercal}
		=S^{-1}H\tilde{S}^{-\intercal}{\Lambda^{(a)}}^\intercal \end{align*} 
	so
	\begin{align*} \vartheta^{(a)}H-(\vartheta^{(a)}S)S^{-1}H-H\tilde{S}^{-\intercal}(\vartheta^{(a)}\tilde{S})^\intercal=H\tilde{S}^{-\intercal}
		{\Lambda^{(a)}}^\intercal\tilde{S}^\intercal H^{-1}H={T^{(a)}}^\intercal H. \end{align*}
\end{proof}
We will denote:
\begin{align*}\phi&\coloneq \phi^{(1)}+\phi^{(2)}, &\tilde{\phi}&\coloneq \tilde{\phi}^{(1)}+\tilde{\phi}^{(2)}.\end{align*}
\begin{pro}
	The following equations hold
\begin{subequations}
		\begin{align}\label{toda:dprin}(\vartheta H)H^{-1}&=\alpha, \\ \label{toda:sub}
		-\phi&=(\Lambda^\intercal)^2\gamma+\Lambda^\intercal\beta,\\ \label{toda:super}
		-\tilde \phi &=\Lambda^\intercal \mathfrak a_- H H^{-1} 
	\end{align}
\end{subequations}
	
\end{pro}
\begin{proof}
	From \eqref{satow1} we obtain
	\begin{align*}
		\vartheta H-\phi H-H\tilde{\phi}^\intercal=TH=(\Lambda^\intercal)^2\gamma H+\Lambda^\intercal\beta H+\alpha H+\mathfrak a_{-}H\Lambda 
	\end{align*}
	Diagonal by diagonal we get the equations.
	% \begin{align*}\underbrace{\vartheta H}_{\text{Diag. Princ.}}-\underbrace{\phi H}_{\text{Estr.Diag.Inf.}}-\underbrace{H\tilde{\phi}^\intercal}_{\text{Estr. Diag. Sup.}}=JH=(\Lambda^\intercal)^2\gamma H+\Lambda^\intercal\beta H+\alpha H+\mathfrak a_{-}H\Lambda \end{align*}
	%\eqref{toda:dprin} can be found equating principal diagonals and \eqref{toda:sub} can be found equating strictly inferior parts.
\end{proof}
\begin{pro}
	The following relations are fulfilled
\begin{subequations}\label{todas}
		\begin{align}
		\label{toda:Sn} \vartheta S^{[n]}&=-(S^{[n-2]}\mathfrak a_{-}^{n-2}\gamma+S^{[n-1]}\mathfrak a_{-}^{n-1}\beta),\\
		\label{toda:S1}\vartheta S^{[1]}&=-\beta,\\
		\label{toda:tildeSn}\vartheta \tilde{S}^{[n]}&=-\mathfrak a_{-}^nH\mathfrak a_{-}^{n-1}H^{-1}\tilde{S}^{[n-1]},\\
		\label{toda:tildeS1}\vartheta \tilde{S}^{[1]}&=-H^{-1}\mathfrak a_{-}H.
	\end{align}
\end{subequations}
\end{pro}
\begin{proof}
	Relations \eqref{toda:Sn} and \eqref{toda:S1} can be obtained from \eqref{toda:sub} if we consider the equation in the form $\vartheta S=-((\Lambda^\top)^2\gamma+\Lambda^\top)S$ equating by diagonals.
	Analogously, \eqref{toda:tildeS1} and \eqref{toda:tildeSn} con be obtained from \eqref{toda:super}.
\end{proof}
\begin{pro}
	For the step-line recursion coefficients the following  expressions hold
	\begin{align}\label{eq:alphabetagamma2}
		\left\{\begin{aligned}
			\alpha=	(\vartheta H )H^{-1}&= \mathfrak a_{+}S^{[1]}-S^{[1]},\\
			\beta=-\vartheta S^{[1]}&=	\mathfrak a_{+}S^{[2]}-S^{[2]}-S^{[1]}\mathfrak a_{-}\big((\vartheta H )H^{-1}\big), \\ 
			\gamma=-\vartheta S^{[2]}+(\vartheta \mathfrak a_{-}S^{[1]})S^{[1]}&=H^{-1}\mathfrak a^2_{-}H,
		\end{aligned}\right.
	\end{align}
\end{pro}
\begin{proof}
	It is proven using Equations \eqref{eq:alphabetagamma} and Proposition \ref{pro:alphabetagamma}, just split \eqref{toda:sub} by diagonals.
\end{proof}

\begin{coro}
	In terms of tau functions the entries in recursion Hessenberg matrix can be expressed as follows
	\begin{align*}
		\alpha_n&=\vartheta \log\frac{\tau_{n+1}}{\tau_n}, &
		\beta_n &=\vartheta^2\log\tau_n, & \gamma_n&=\frac{\tau_{n+3}\tau_n}{\tau_{n+2}\tau_{n+1}}.
	\end{align*}
\end{coro}
\begin{proof}
	It follows from Equations \eqref{eq:alphabetagamma2} and \eqref{eq:H_p_tau}. 
\end{proof}

\begin{pro}[Multiple Toda equations]
	The functions $q_n\coloneq\log H_n$ and $f_n\coloneq S^{[1]}_n=p^1_{n+1}$  solve the system
	\begin{align}\label{eq:2-toda}
		\left\{\begin{aligned}
			\vartheta q_n&=f_{n-1}-f_n,\\
			\vartheta^2 f_n-(2f_n-f_{n+1}-f_{n-1})\vartheta f_n&=\Exp{q_{n+1}-q_{n-1}}-\Exp{q_{n+2}-q_n}.
		\end{aligned}\right.
	\end{align}
\end{pro}

\begin{proof}
	Using  Equations \eqref{eq:alphabetagamma2} we write
	\begin{align*}
		\vartheta S^{[1]}&=	-\mathfrak a_{+}S^{[2]}+S^{[2]}+S^{[1]}(\vartheta \mathfrak a_{-}H )(\mathfrak a_{-}H)^{-1}, &
		\vartheta S^{[2]}&=(\vartheta \mathfrak a_{-}S^{[1]})S^{[1]}-H^{-1}\mathfrak a^2_{-}H.
	\end{align*}
	Acting with $\vartheta$ on the first equation an using then the second equation  we get
	\begin{align*}
		(\vartheta \mathfrak a_{-}S^{[1]})S^{[1]}-(\vartheta S^{[1]})\mathfrak a_{+}S^{[1]}+\mathfrak a_{+}H^{-1}\mathfrak a_{-}H-H^{-1}\mathfrak a^2_{-}H=\vartheta^2 S^{[1]}-\vartheta\left(S^{[1]}\big(S^{[1]}-\mathfrak a_-S^{[1]}\big)\right)
	\end{align*}

	Then, the equation reads
	\begin{align*}
		f_{n-1}	\vartheta f_n-f_n\vartheta f_{n+1}+\Exp{q_{n+2}-q_n}-\Exp{q_{n+1}-q_{n-1}}&=-\vartheta^2 f_n+(f_n-f_{n+1})\vartheta f_n+f_n(\vartheta f_n-\vartheta f_{n+1})
	\end{align*}
	that after some clearing leads to Equations \eqref{eq:2-toda}.
\end{proof}

\begin{teo} \label{theo:third_order_PDE}
	The $\tau$-function satisfies the following nonlinear third order differential--difference equation
	\begin{align}\label{eq:multi_toda}
		\vartheta^3\tau_{n+1}
		-
		\left(
		\frac{\vartheta\tau_{n+2}}{\tau_{n+2}}+\frac{\vartheta\tau_{n+1}}{\tau_{n+1}}+\frac{\vartheta\tau_{n}}{\tau_{n}}\right)	\vartheta^2\tau_{n+1}+\frac{(\vartheta \tau_{n+1})^2}{\tau_{n+1}}
		\left(\frac{\vartheta\tau_{n+2}}{\tau_{n+2}}-\frac{\vartheta\tau_{n}}{\tau_{n}}\right)=\frac{\tau_{n+3}\tau_{n}}{\tau_{n+2}}-
		\frac{\tau_{n+2}\tau_{n-1}}{\tau_n}.
	\end{align}
\end{teo}

\begin{proof}
	Notice that the first equation $\vartheta q_n=f_{n-1}-f_n$, when the $\tau$ function is used, becomes an identity, indeed
	\begin{align*}
		\vartheta q_n&=\vartheta \log \tau_{n+1}-\vartheta \log \tau_n, & f_{n-1}-f_n&=p^1_n-p^1_{n+1}=-\vartheta \log\tau_n+\vartheta\log\tau_{n+1}.
	\end{align*}
	Then, the second equation reads
	\begin{align*}
		-\vartheta^3 \log\tau_{n+1}-\vartheta(2\log\tau_{n+1}-\log\tau_{n+2}-\log\tau_{n})\vartheta^2 \log\tau_{n+1}&=\frac{\tau_{n+2}\tau_{n-1}}{\tau_{n+1}\tau_n}-\frac{\tau_{n+3}\tau_{n}}{\tau_{n+2}\tau_{n+1}},
		%\Exp{q_{n+2}-q_n}.
	\end{align*}
	that expanding the differentiations leads to 
	\begin{align*}
		\frac{2(\vartheta \tau_{n+1})^3}{\tau_{n+1}^3}-3\frac{(\vartheta\tau_{n+1})\vartheta^2\tau_{n+1}}{\tau_{n+1}^2}+\frac{\vartheta^3\tau_{n+1}}{\tau_{n+1}}
		+\left(	-\frac{(\vartheta \tau_{n+1})^2}{\tau_{n+1}^2}+
		\frac{\vartheta^2\tau_{n+1}}{\tau_{n+1}}
		\right)\left(2\frac{\vartheta\tau_{n+1}}{\tau_{n+1}}-\frac{\vartheta\tau_{n+2}}{\tau_{n+2}}-\frac{\vartheta\tau_{n}}{\tau_{n}}\right)=\frac{\tau_{n+3}\tau_{n}}{\tau_{n+2}\tau_{n+1}}-\frac{\tau_{n+2}\tau_{n-1}}{\tau_{n+1}\tau_n},
	\end{align*}
	and after some clearing we get 
	\begin{align*}
		-\frac{(\vartheta\tau_{n+1})\vartheta^2\tau_{n+1}}{\tau_{n+1}^2}+\frac{\vartheta^3\tau_{n+1}}{\tau_{n+1}}
		+\left(	-\frac{(\vartheta \tau_{n+1})^2}{\tau_{n+1}^2}+
		\frac{\vartheta^2\tau_{n+1}}{\tau_{n+1}}
		\right)\left(-\frac{\vartheta\tau_{n+2}}{\tau_{n+2}}-\frac{\vartheta\tau_{n}}{\tau_{n}}\right)=\frac{\tau_{n+3}\tau_{n}}{\tau_{n+2}\tau_{n+1}}-\frac{\tau_{n+2}\tau_{n-1}}{\tau_{n+1}\tau_n},
	\end{align*}
	from where we immediately   find Equation \eqref{eq:multi_toda}.
\end{proof}

%\begin{rem}
%	Notice that   solutions for Equation \eqref{eq:multi_toda} are double Wronskians with respect the derivative $\vartheta=\frac{\d}{\d\rho }$, see Lemma \ref{lem:theta-rho}.
%	\begin{equation*}
%		\tau_n=\mathscr W_n(f_1(e^{\rho+\tilde{\rho}}), f_2(e^{\rho-\tilde{\rho}})).
%	\end{equation*}
%\end{rem}

\begin{rem}
	For the discrete hypergeometric orthogonal polynomials the standard Toda equation, see for example \cite[Equation (168)]{Manas_Fernandez-Irisarri},   can be written in terms of the $\tau$-function as the following nonlinear second order differential--difference equation
	\begin{align*}
		\vartheta^2\tau_{n+1}-\frac{(\vartheta\tau_{n+1})^2}{\tau_{n+1}}=\frac{\tau_{n+2}\tau_{n}}{\tau_{n+1}}.
	\end{align*}
	We clearly see that \eqref{eq:multi_toda} is an extension to the  multiple orthogonal realm  of the standard Toda equation.
\end{rem}

%\begin{align*}
%	\frac{2 \, \frac{\partial}{\partial x}f\left(x\right)^{3}}{f\left(x\right)^{3}} - \frac{3 \, \frac{\partial}{\partial x}f\left(x\right) \frac{\partial^{2}}{(\partial x)^{2}}f\left(x\right)}{f\left(x\right)^{2}} + \frac{\frac{\partial^{3}}{(\partial x)^{3}}f\left(x\right)}{f\left(x\right)},\\
%	-\frac{\frac{\partial}{\partial x}f\left(x\right)^{2}}{f\left(x\right)^{2}} + \frac{\frac{\partial^{2}}{(\partial x)^{2}}f\left(x\right)}{f\left(x\right)}
%\end{align*}

% \begin{rem}If we equalize strictly superdiagonal parts we obtain 

	% This equation we will be useful for future calculations.
	% \end{rem}
Next we write this multiple Toda equations as the following  Toda system for the recursion coefficients $\alpha$, $\beta$ and $\gamma$:
\begin{pro}
	For multiple orthogonality with two weights we have the following Toda type system:
	\begin{align}\label{ts:alpha}\vartheta \alpha&=\beta-\mathfrak a_{+}\beta, \\ \label{ts:beta} \vartheta \beta&=\gamma-\mathfrak a_{+}\gamma+\beta(\mathfrak a_{-}\alpha-\alpha), \\ \label{ts:gamma} \vartheta \gamma&=\gamma(\mathfrak a_{-}^2\alpha-\alpha). \end{align}
	% and the following is true too:
	% \begin{equation}\label{ts:H}\vartheta H=\alpha H.\end{equation}
\end{pro}
\begin{proof}
	Equation \eqref{ts:alpha} is immediately obtained from \eqref{toda:dprin}.
	From \eqref{toda:sub}, analyzing diagonal by diagonal we can obtain two different equations:
	\begin{align*}\vartheta S^{[1]}&=-\beta, \\ \gamma&=-\vartheta S^{[2]}+(\vartheta \mathfrak a_{-}S^{[1]})S^{[1]},\end{align*}
	using equations above and expressions for $\alpha$, $\beta$ and $\gamma$ depending of $S^{[2]}$ and $S^{[1]}$ we obtain \eqref{ts:alpha}, \eqref{ts:beta} and \eqref{ts:gamma}.
\end{proof}
\begin{pro} [Lax pair] The matrices $T$ and $\phi$ are a Lax pair, i.e. the following Lax equation is satisfied 
	\begin{equation}\label{lax}\vartheta T=[\phi,	T]=[T_{+},T],
	\end{equation}
	 where 
	\begin{align*}
		T_{-}&=(\Lambda^\intercal)^2\gamma+\Lambda^\intercal \beta, \\ 
		T_{+}&=T-T_{-}=\alpha+\Lambda.
	\end{align*}
	
\end{pro}
\begin{proof}We  have $T=S\Lambda S^{-1}$ so that
	\begin{align*}
		\vartheta^{(a)}T=(\vartheta^{(a)}S)\Lambda S^{-1}-S\Lambda S^{-1}(\vartheta^{(a)}S)S^{-1}=\underbrace{(\vartheta^{(a)}S)S^{-1}}_{\phi^{(a)}}\underbrace{S\Lambda S^{-1}}_{T}-\underbrace{S\Lambda S^{-1}}_{T}\underbrace{(\vartheta^{(a)}S)S^{-1}}_{\phi^{(a)}}\end{align*}
	and we get Equation  \eqref{lax}.
	The other form can be proved with \eqref{toda:sub}
	and:
	\begin{align*}[-T_{-},T]=[T_{+}-T,T]=[T_{+},T].\end{align*}
\end{proof}
\begin{rem}\label{rem:extending_tau2}
As mentioned in Remark \ref{rem:extending_tau}, the previous results are applicable to more general situations beyond generalized hypergeometric series. The only requirement is that $\vartheta w^{(a)}(k)=kw^{(a)}(k)$. Once again, we utilize the tau function defined as $\tau_n=\mathscr W_n[\rho^{(1)}_0,\rho^{(2)}_0]$, where $\rho^{(a)}_0=\sum_{n=0}^\infty w^{(a)}(k)$.
\end{rem}

Coming back to the hypergeometric situation, there is another compatibility equation for Laguerre-Freud matrix, when we make calculations for concrete cases of multiple orthogonal polynomials we will name it \textit{Compatibility II}.
\begin{pro}
	The following compatibility conditions hold for recursion and Laguerre-Freud matrices
	\begin{align}\label{comII:Psi}\vartheta \Psi&=[\phi,\Psi]=[-T_{-},\Psi]=[\Psi,T_{-}], \\\label{compII:PsiT} \vartheta \Psi^\intercal&=\Psi^\intercal+[\mu,\Psi^\intercal]=\Psi^\intercal+[-T_{+}^\intercal,\Psi^\intercal],\end{align}
	with lower triangular matrices  $\mu\coloneq\mu^{(1)}+\mu^{(2)}$ and $\mu^{(a)}\coloneq(\vartheta_{(a)}H^{-1})H+H^{-1}\tilde{\phi}^{(a)}H.$
	
	%\begin{align*}\mu^{(a)}=(\mu^{(a)})^{[0]}+\Lambda^\intercal(\mu^{(a)})^{[1]}+\cdots.\end{align*}
\end{pro}

\begin{proof}
	We have that 
	\begin{align*} 
		{} \theta(z)B(z-1)&=\Psi B(z),\\ 
		\vartheta B(z)&=\phi B(z).
	\end{align*}
	Therefore,
	\begin{align*}
		\vartheta(\theta(z)B(z-1))=\vartheta(\Psi B(z))=(\vartheta\Psi)B(z)+\Psi\vartheta B(z)=((\vartheta\Psi)-\Psi\phi)B(z),
	\end{align*}
	then
	\begin{align*}
		\vartheta(\theta(z)B(z-1))=((\vartheta\Psi)+\Psi\phi)B(z).
	\end{align*}
	We also have 
	\begin{equation*}
		\vartheta(\theta(z)B(z-1))=\theta(z)\phi B(z-1)=\phi\Psi B(z).
	\end{equation*}
	Equating  above relations and grouping terms and using  Lemma \ref{lemma:thecnical_lemma_X} we find that
	\begin{align*}
		\vartheta \Psi=[\phi,\Psi].
	\end{align*}
	To prove \eqref{compII:PsiT} we need to show  first $\vartheta^{(1)}A^{(1)}=\mu^{(1)}A^{(1)}$, $\vartheta^{(2)}A^{(1)}=\mu^{(2)}A^{(1)}$, $\vartheta^{(1)}A^{(2)}=\mu^{(1)}A^{(2)}$ and $\vartheta^{(2)}A^{(2)}=\mu^{(2)}A^{(2)}$.We will prove equations for $A^{(1)}$ because for $A^{(2)}$ it is analogous:
	\begin{multline*}
		\vartheta^{(2)}A^{(1)}=\vartheta^{(2)}(H^{-1}\tilde{S})X^{(1)}=((\vartheta^{(2)}H^{-1})\tilde{S}+H^{-1}(\vartheta^{(2)}\tilde{S}))X^{(1)}=(\vartheta^{(2)}H^{-1})HH^{-1}\tilde{S}X^{(1)}\\+H^{-1}(\vartheta^{(2)}\tilde{S})(H^{-1}\tilde{S})^{-1}(H^{-1}\tilde{S})X^{(1)} =((\vartheta^{(2)}H^{-1})H+H^{-1}(\vartheta^{(2)}\tilde{S})\tilde{S}^{-1}H)A^{(1)}.
	\end{multline*}
	Hence, we get \begin{align*}\vartheta^{(2)}A^{(1)}=\underbrace{((\vartheta^{(2)}H^{-1})H+H^{-1}\tilde{\phi}^{(2)}H)}_{\mu^{(2)}}A^{(1)},
	\end{align*}
	and
	\begin{multline*}\vartheta^{(1)}A^{(1)}=\vartheta^{(1)}(H^{-1}\tilde{S})X^{(1)}=((\vartheta^{(1)}H^{-1})\tilde{S}+H^{-1}(\vartheta^{(1)}\tilde{S}))X^{(1)}
		=(\vartheta^{(1)}H^{-1})HH^{-1}\tilde{S}X^{(1)}\\+H^{-1}(\vartheta^{(1)}\tilde{S})(H^{-1}\tilde{S})^{-1}(H^{-1}\tilde{S})X^{(1)} =((\vartheta^{(1)}H^{-1})H+H^{-1}(\vartheta^{(1)}\tilde{S})\tilde{S}^{-1}H)A^{(1)}.
	\end{multline*}
	Consequently,
	\begin{align*}
		\vartheta^{(1)}A^{(1)}=\underbrace{((\vartheta^{(1)}H^{-1})H+H^{-1}\tilde{\phi}^{(1)}H)}_{\mu^{(1)}}A^{(1)},
	\end{align*}
	if we add both last equations we have $\vartheta A^{(1)}=\mu A^{(1)}$ (and for $A^{(2)}$ there is the same equation).
	The compatibility Toda--Pearson must hold:
	\begin{align*}
		\sigma^{(1)}(z)A^{(1)}(z+1)&=\Psi^\intercal A^{(1)}(z),\\ 
		\vartheta A^{(1)}(z)&=\mu A^{(1)}(z),
	\end{align*}
	On the one hand
	\begin{align*}
		\vartheta(\sigma^{(1)}(z)A^{(1)}(z+1))=\vartheta(\Psi^\intercal A^{(1)}(z))=\vartheta(\Psi^\intercal)A^{(1)}(z)+\Psi^\intercal\vartheta A^{(1)}(z)=\vartheta(\Psi^\intercal)A^{(1)}(z)+\Psi^\intercal\mu A^{(1)}(z).
	\end{align*}
	On the other hand
	\begin{align*}
		\vartheta(\sigma^{(1)}(z)A^{(1)}(z+1))=\sigma^{(1)}(z)A^{(1)}(z+1)+\sigma^{(1)}(z)\mu A^{(1)}(z+1)=(I+\mu)\Psi^\intercal A^{(1)}(z).
	\end{align*} 
	Consequently, equating both:
	\begin{align*}
		\vartheta(\Psi^\intercal)A^{(1)}(z)+\Psi^\intercal\mu A^{(1)}(z)=(I+\mu)\Psi^\intercal A^{(1)}(z),
	\end{align*}
	and using Lemma \ref{lem:tecnical_typeI} we finally get
	\begin{align*}
		\vartheta(\Psi^\intercal)=\Psi^\intercal+[\mu,\Psi^\intercal].
	\end{align*}
\end{proof}

\subsection{Multiple Nijhoff--Capel discrete Toda equations}
	
We will now study the compatibility conditions between shifts in the three families of hypergeometric parameters:
\begin{align*}
	\lbrace b_i^{(1)}\rbrace^{M^{(1)}}_{i=1}, 
	\lbrace b_i^{(2)}\rbrace^{M^{(2)}}_{i=1}, \lbrace c_j\rbrace^N_{j=1},
\end{align*}
for functions $u_n$. We will use the following notation:
\begin{enumerate}
	\item $\hat u$ denotes a shift in one and only one parameter in the first family of parameters, i.e., $b^{(1)}_r\to b^{(1)}_r+1$ for only one given $r\in\{1,\dots,M^{(1)}\}$. The corresponding tridiagonal connection matrix is denoted by $\Omega^{(r)}$, and we define $\hat d\coloneq b^{(1)}_r$.
	\item $\check u$ denotes a shift in one and only one parameter in the second family of parameters, i.e., $b^{(2)}_s\to b^{(2)}_s+1$ for only one given $q\in\{1,\dots,M^{(1)}\}$. The corresponding tridiagonal connection matrix is denoted by $\Omega^{(q)}$, and we define $\check  d\coloneq b^{(2)}_q$.
	\item $\tilde u$ denotes a shift in one and only one parameter in the third family of parameters, i.e., $c_s\to c_s-1$ for only one given $s\in\{1,\dots,N\}$. The corresponding tridiagonal connection matrix is denoted by $\Omega^{(s)}$, and we define $\tilde  d\coloneq c_s-1$.
	\item $\bar u$ denotes the shift $n\to n+2$.
\end{enumerate}
\begin{lemma}Connection matrices fulfill the following compatibility conditions:
\begin{subequations}
	\begin{align}\label{compsr}\Omega^{(s)}\tilde{\Omega}^{(r)}=\Omega^{(r)}\hat{\Omega}^{(s)}, \\
\label{compsq}\Omega^{(s)}\tilde{\Omega}^{(q)}=\Omega^{(q)}\check{\Omega}^{(s)},\\
\label{comprq}\Omega^{(r)}\hat{\Omega}^{(q)}=\Omega^{(q)}\check{\Omega}^{(r)}.
\end{align}
\end{subequations}
\end{lemma}
\begin{proof}
In the one hand, connection formulas can be written as follows
\begin{align*}\Omega^{(r)}\hat{B}&=B, & \Omega^{(s)}\tilde{B}=B, \end{align*}
so that $\tilde{\Omega}^{(r)}\tilde{\hat{B}}=\tilde{B} $ and, consequently, 
$\Omega^{(s)}\tilde{\Omega}^{(r)}\tilde{\hat{B}}=\Omega^{(s)}\tilde{B}=B$.
On the other hand, we have
$\hat{\Omega}^{(s)}\hat{\tilde{B}}=\hat{B}$ so that 
$\Omega^{(r)}\hat{\Omega}^{(s)}\hat{\tilde{B}}=\Omega^{(r)}\hat{B}=B$. 
Hence, the compatibility $\tilde{\hat{P}}=\hat{\tilde{P}}$ leads to \eqref{compsr}.
Relations \eqref{compsq} and \eqref{comprq} are proven analogously.
\end{proof}
From this point we will use the following notation:
\begin{align*}
	H_{2n}&\eqcolon u, & H_{2n+1}&\eqcolon v, & S_{2n}^{[1]}&\eqcolon f,  & S_{2n+1}^{[1]}&\eqcolon g, & S_{2n}^{[2]}&\eqcolon F, & S_{2n+1}^{[2]}&\eqcolon G. 
\end{align*}
 \begin{teo}[Multiple Nijhoff--Capel Toda systems]
	The four  functions $u,v,f$ and $g$ satisfy the following 3D system of four  nonlinear difference  equations:
\begin{enumerate}
	\item 	 For the shifts {}$\hat{}$, {}$ \tilde{}$ and {} $\bar{}$:
\begin{subequations}\label{compsr-}	\begin{align}\label{compsr:p1d2}
 g(\tilde{f}-\hat{f})+\tilde{g}(\tilde{\hat{f}}-\tilde{f})+\hat{g}(\hat{f}-\hat{\tilde{f}})&=	\frac{1}{\tilde{d}}\left(\frac{\hat{\bar {u}}}{\hat{\tilde{u}}}-\frac{\bar {u}}{\tilde{u}}\right)+
 \frac{1}{\hat{d}}\left(\frac{\bar {u}}{\hat{u}}-\frac{\tilde{\bar {u}}}{\tilde{\hat{u}}}\right),\\
 \label{compsr:p2d2}
\bar {f}(\hat{g}-\tilde{g})+\tilde{\bar{f}}(\tilde{g}-\tilde{\hat{g}})+\hat{\bar {f}}(\hat{\tilde{g}}-\hat{g})&= \frac{1}{\tilde{d}}\left(\frac{\bar {v}}{\tilde{v}}-\frac{\hat{\bar {v}}}{\hat{\tilde{v}}}\right),\\ \label{compsr:p1d3}
 \frac{1}{\tilde{d}}\frac{\hat{\bar {u}}}{\hat{\tilde{u}}}(\hat{\bar {f}}-\bar {f})+\frac{1}{\hat{d}}\frac{\tilde{\bar {u}}}{\tilde{\hat{u}}}(\bar {f}-\tilde{\bar {f}})+\frac{1}{\tilde{d}}\frac{\bar {v}}{\tilde{v}}(\tilde{f}-\tilde{\hat{f}})&=0,\\ \label{compsr:p2d3}
  \frac{1}{\hat{d}}\frac{\bar {\bar {u}}}{\hat{\bar {u}}}(\hat{g}-\hat{\tilde{g}})+\frac{1}{\tilde{d}}\frac{\bar {\bar {u}}}{\tilde{\bar {u}}}(\tilde{\hat{g}}-\tilde{g})+\frac{1}{\tilde{d}}\frac{\hat{\bar {v}}}{\hat{\tilde{v}}}(\bar {g}-\hat{\bar {g}})&=0.
\end{align}
\end{subequations}
\item For the shifts {}$\check{}$, {}$ \tilde{}$ and {} $\bar{}$:
\begin{subequations}\label{compsq-}	
	 \begin{align}\label{compsq:p1d2}
		g(\tilde{f}-\check{f})+\tilde{g}(\check{\tilde{f}}-\tilde{f})+\check{g}(\check{f}-\check{\tilde{f}})&=	\frac{1}{\tilde{d}}\left(\frac{\check{\bar {u}}}{\check{\tilde{u}}}-\frac{\bar {u}}{\tilde{u}}\right)\\
		\label{compsq:p2d2} 
		\bar {f}(\tilde{g}-\check{g})+\tilde{\bar {f}}(\tilde{\check{g}}-\tilde{g})+\check{\bar {f}}(\check{g}-\check{\tilde{g}})&= \frac{1}{\tilde{d}}\left(\frac{\check{\bar {v}}}{\check{\tilde{v}}}-\frac{\bar {v}}{\tilde{v}}\right)+\frac{1}{\check{d}}\left(\frac{\bar {v}}{\check{v}}-\frac{\tilde{\bar {v}}}{\check{\tilde{v}}}\right),\\ \label{compsq:p1d3}\frac{1}{\tilde{d}}\frac{\check{\bar {u}}}{\check{\tilde{u}}}(\bar {f}-\check{\bar {f}})+\frac{1}{\check{d}}\frac{\bar {v}}{\check{v}}(\check{f}-\check{\tilde{f}})+\frac{1}{\tilde{d}}\frac{\bar {v}}{\tilde{v}}(\tilde{\check{f}}-\tilde{f})&=0,\\ \label{compsq:p2d3}
		\frac{1}{\tilde{d}}\frac{\bar {\bar {u}}}{\tilde{\bar {u}}}(\tilde{g}-\tilde{\check{g}}) +\frac{1}{\check{d}}\frac{\tilde{\bar {v}}}{\tilde{\check{v}}}(\bar {g}-\tilde{\bar {g}})+\frac{1}{\tilde{d}}\frac{\check{\bar {v}}}{\check{\tilde{v}}}(\check{\bar {g}}-\bar {g})&=0.\end{align}
	\end{subequations}
\item For the shifts {}$\hat{}$, {}$ \check{}$ and {} $\bar{}$:
\begin{subequations}\label{compabd3p}
	   \begin{align}
		\label{compabd3p1}
		\frac{1}{\hat d}\frac{\check{\bar {u}}}{\check{\hat{u}}}(\bar {f}-\check{\bar {f}})+       \frac{1}{\check d}\frac{\bar {v}}{\check{v}}(\check{f}-\check{\hat{f}})&=0,\\
		\label{compabd3p2}
		\frac{1}{\hat d}\frac{\bar {\bar {u}}}{\hat{\bar {u}}}(\hat{\check{g}}-\hat{g})+    \frac{1}{\check d}\frac{\hat{\bar {v}}}{\check{\hat{v}}}(\hat{\bar {g}}-\bar {g})&= 0,\\
		\label{compabd2p1}
		g(\check{f}-\hat{f})+\check{g}(\check{\hat{f}}-\check{f})+\hat{g}(\hat{f}-\check{\hat{f}})&=\frac{1}{\hat{d}}\left(\frac{\bar {u}}{\hat{u}}-\frac{\check{\bar {u}}}{\hat{\check{u}}}\right),\\
		\label{compabd2p2}
		\bar {f}(\check{g}-\hat{g})+\check{\bar {f}}(\check{\hat{g}}-{\check{g}})+\hat{\bar {f}}(\hat{g}-\check{\hat{g}})&=  \frac{1}{\check{d}}\left(\frac{\hat{\bar {v}}}{\check{\hat{v}}}-\frac{\bar {v}}{\check{v}}\right).
	\end{align}
\end{subequations}
\end{enumerate}
 \end{teo}
 \begin{proof}
\begin{enumerate}
	\item   From \eqref{compsr} %using \eqref{est:iO} for this case 
 we can find the following two non-trivial equations for matrices $S^{[1]}$ and $H$ from second and third subdiagonals:
 \begin{multline}\label{compsr:d2}
 	\frac{1}{\tilde{d}}(\mathfrak a_{-}^2\hat{H})\hat{\tilde{H}}^{-1}+(\mathfrak a_{-}S^{[1]}-\mathfrak a_{-}\hat{S}^{[1]})(\hat{S}^{[1]}-\hat{\tilde{S}}^{[1]})+\frac{1}{\hat{d}}I^{(1)}\hat{H}^{-1}\mathfrak a_{-}^2H\\=\frac{1}{\hat{d}}\tilde{\hat{H}}^{-1}(\mathfrak a_{-}^2\tilde{H})I^{(1)}+(\mathfrak a_{-}S^{[1]}-\mathfrak a_{-}\hat{S}^{[1]})(\tilde{S}^{[1]}-\tilde{\hat{S}}^{[1]})+\frac{1}{\tilde{d}}(\mathfrak a_{-}^2H)\tilde{H}^{-1}
 \end{multline}
 and
 \begin{multline}\label{compsr:d3}
 	\frac{1}{\tilde{d}}(\mathfrak a_{-}^2\hat{H})\hat{\tilde{H}}^{-1}(\mathfrak a_{-}^2S^{[1]}-\mathfrak a_{-}^2\hat{S}^{[1]})+\frac{1}{\hat{d}}I^{(2)}(\mathfrak a_{-}^3H)(\mathfrak a_{-}\hat{H}^{-1})(\hat{S}^{[1]}-\tilde{\hat{S}}^{[1]})\\=\frac{1}{\hat{d}}I^{(1)}(\mathfrak a_{-}^2\tilde{H})\tilde{\hat{H}}^{-1}+\frac{1}{\tilde{d}}(\mathfrak a_{-}^3H)(\mathfrak a_{-}\tilde{H}^{-1})(\tilde{S}^{[1]}-\tilde{\hat{S}}^{[1]}).
 \end{multline}
 From \eqref{compsr:d2}, equating by one side the terms that multiplied $I^{(1)}$ we obtain \eqref{compsr:p1d2} and from the terms multiplied by $I^{(2)}$ it is obtained \eqref{compsr:p2d2}. 
 From \eqref{compsr:d3}, equating terms in $I^{(1)}$ it is obtained \eqref{compsr:p1d3} and from terms in $I^{(2)}$ we get \eqref{compsr:p2d3}.
 \item
 From \eqref{compsq}, we  find two non-trivial equations for matrices $S^{[1]}$ and $H$ from second and third subdiagonals:
 \begin{multline}\label{compsq:d2}
 	\frac{1}{\tilde{d}}(\mathfrak a_{-}^2\check{H})\check{\tilde{H}}^{-1}+(\mathfrak a_{-}S^{[1]}-\mathfrak a_{-}\check{S}^{[1]})(\check{S}^{[1]}-\check{\tilde{S}}^{[1]})+\frac{1}{\check{d}}I^{(2)}\check{H}^{-1}\mathfrak a_{-}^2H\\=\frac{1}{\check{d}}I^{(2)}\tilde{\check{H}}^{-1}\mathfrak a_{-}^2\tilde{H}+(\mathfrak a_{-}S^{[1]}-\mathfrak a_{-}\tilde{S}^{[1]})(\tilde{S}^{[1]}-\tilde{\check{S}}^{[1]})+\frac{1}{\tilde{d}}(\mathfrak a_{-}^2H)\tilde{H}^{-1}
 \end{multline}
 and
 \begin{multline}\label{compsq:d3}\frac{1}{\tilde{d}}(\mathfrak a_{-}^2\check{H})\check{\tilde{H}}^{-1}(\mathfrak a_{-}^2S^{[1]}-\mathfrak a_{-}^2\check{S}^{[1]})+\frac{1}{\check{d}}I^{(1)}(\mathfrak a_{-}\check{H}^{-1})(\mathfrak a_{-}^3H)(\check{S}^{[1]}-\check{\tilde{S}}^{[1]})\\=\frac{1}{\check{d}}I^{(2)}\tilde{\check{H}}^{-1}\mathfrak a_{-}^2\tilde{H}+\frac{1}{\tilde{d}}(\mathfrak a_{-}^3H)(\mathfrak a_{-}\tilde{H}^{-1})(\tilde{S}^{[1]}-\tilde{\check{S}}^{[1]}).
 \end{multline}
 From \eqref{compsq:d2}, equating by one side the terms that multiplied $I^{(1)}$ we obtain \eqref{compsq:p1d2} and from the terms multiplied by $I^{(2)}$ it is obtained \eqref{compsr:p2d2}. 
 From Equation \eqref{compsq:d3}, equating terms in $I^{(1)}$ it is obtained \eqref{compsq:p1d3} and from terms in $I^{(2)}$ we obtain \eqref{compsq:p2d3}.
\item 
    From \eqref{comprq}, from third subdiagonal we can obtain
    \begin{multline}\frac{1}{\tilde{d}}(\mathfrak a_{-}^2\tilde{H}^{-1})I^{(1)}(\mathfrak a_{-}^3H)(\tilde{S}^{[1]}-\tilde{\hat{S}}^{[1]})+\frac{1}{\hat{d}}I^{(1)}\tilde{\hat{H}}^{-1}(\mathfrak a_{-}^2\tilde{H})(\mathfrak a_{-}^2S^{[1]}-\mathfrak a_{-}^2\tilde{S}^{[1]})\\=\frac{1}{\hat{d}}I^{(2)}(\mathfrak a_{-}\hat{H}^{-1})(\mathfrak a_{-}^3H)(\hat{S}^{[1]}-\hat{\tilde{S}}^{[1]})+\frac{1}{\tilde{d}}I^{(2)}\hat{\tilde{H}}^{-1}\mathfrak a_{-}^2\hat{H}\end{multline}
 and from the terms of $I^{(1)}$ we can obtain \eqref{compabd3p1}, and from terms of $I^{(2)}$ we can obtain \eqref{compabd3p2}.\\
 From second subdiagonal we get the following equation:
 \begin{multline}
   \frac{1}{  \tilde{d}}I^{(2)}\tilde{H}^{-1}\mathfrak a_{-}^2H+(\mathfrak a_{-}S^{[1]})\tilde{S}^{[1]}-(\mathfrak a_{-}\tilde{S}^{[1]})(\tilde{S}^{[1]}-\tilde{\hat{S}}^{[1]})+\frac{1}{\hat{d}}I^{(1)}\tilde{\hat{H}}^{-1}\mathfrak a_{-}^2\tilde{H}\\
  =\frac{1}{\hat{d}}I^{(1)}\hat{H}^{-1}\mathfrak a_{-}^2H+(\mathfrak a_{-}S^{[1]})\hat{S}^{[1]}-(\mathfrak a_{-}\hat{S}^{[1]})(\hat{S}^{[1]}-\hat{\tilde{S}}^{[1]})+\frac{1}{\tilde{d}}I^{(2)}\hat{\tilde{H}}^{-1}\mathfrak a_{-}^2\hat{H},
 \end{multline}
from terms of $I^{(1)}$ we obtain \eqref{compabd2p1} and from terms of $I^{(2)}$ we deduce \eqref{compabd2p2}. 
\end{enumerate}
 \end{proof}
\begin{rem}
Each set of Equations \eqref{compsr-},  \eqref{compsq-} and \eqref{compabd3p}	is a nonlinear system of four equations of discrete partial difference equations in three variables for four functions. The first two systems  \eqref{compsr-} and  \eqref{compsq-} are equivalent under the interchange of $\hat{}\longleftrightarrow\check{}$, $u\longleftrightarrow  v$  and $f\longleftrightarrow  g$.  
\end{rem}
\begin{rem}
	Notice that Equations \eqref{compsr-} and \eqref{compsq-} can be considered for the AT systems of weights  described Lemma \ref{lemma:ATSystems}. Therefore, we are sure that the multiple orthogonal polynomials do exist. For  other cases, one needs to proof perfectness or that the system is AT, which is still a open issue.
\end{rem}
\begin{rem}
In \cite{Manas_Fernandez-Irisarri}, we discovered the Nijhoff--Capel Toda equation (Equation (128)) for standard non-multiple hypergeometric orthogonal polynomials \cite{Manas_Fernandez-Irisarri}. It can be expressed as:
\begin{align*}
	\frac{\bar u}{\tilde u}-\frac{\hat{\bar u}}{\hat{\tilde  u}}=\frac{\bar u}{\hat u}-\frac{\tilde{\bar u}}{\tilde{\hat u}}
\end{align*}
This equation is obtained by considering Equation \eqref{compsr:p1d2} with $v=g=0$. Consequently, we refer to these systems as multiple Nijhoff--Capel Toda systems, namely \eqref{compsr-}, \eqref{compsq-}, and \eqref{compabd3p}. The mentioned equation is the fifth equation among the canonical types for consistency equations corresponding to an octahedral sub-lattice, as described in \cite{Adler} and \cite[Equation (3.107c) and \S 3.9]{Hietarinta}.
\end{rem}

\begin{teo}\label{rem:tau_Nijhoff_Capel}
The multiple Nijhoff--Capel system \eqref{compsq-}	 reads
	\begin{align*}
		&	\begin{multlined}[t][\textwidth]	\frac{\tau^1_{2n+1}}{\tau_{2n+1}}\left(\frac{\tilde \tau^1_{2n}}{\tilde \tau_{2n}}-\frac{\hat \tau^1_{2n}}{\hat \tau_{2n}}\right)+\frac{\tilde \tau^1_{2n+1}}{\tilde {{\tau}}_{2n+1}}\left(\frac{\tilde {\hat \tau}^1_{2n}}{\tilde{ \hat \tau}_{2n}}-\frac{\tilde \tau^1_{2n}}{\tilde \tau_{2n}}\right)+\frac{\hat \tau^1_{2n+1}}{\hat {{\tau}}_{2n+1}}\left(\frac{ {\hat \tau}^1_{2n}}{{ \hat \tau}_{2n}}-\frac{\hat{\tilde \tau}^1_{2n}}{\hat{\tilde \tau}_{2n}}\right)\\
			=	\frac{1}{\tilde{d}}\left(
			\frac{\hat {\bar\tau}_{2n+1}}{\hat {\bar\tau}_{2n}}\frac{\hat{\tilde{\tau}}_{2n}}	{\hat{\tilde{\tau}}_{2n+1}}-	\frac{ {\bar\tau}_{2n+1}}{ {\bar\tau}_{2n}}	\frac{\tilde \tau_{2n}}{\tilde \tau_{2n+1}}\right)+
			\frac{1}{\hat{d}}\left(
			\frac{ {\bar\tau}_{2n+1}}{ {\bar\tau}_{2n}}\frac{\hat{{\tau}}_{2n}}	{\hat{{\tau}}_{2n+1}}-	\frac{ \tilde {\bar\tau}_{2n+1}}{ \tilde{\bar\tau}_{2n}}	\frac{\tilde {\hat \tau}_{2n}}{\tilde {\hat \tau}_{2n+1}}\right),
		\end{multlined}\\
		&		\frac{{\bar\tau}^1_{2n}}{{\bar\tau}_{2n}}\left(\frac{\tilde \tau^1_{2n+1}}{\tilde \tau_{2n+1}}+\frac{\hat \tau^1_{2n+1}}{\hat \tau_{2n+1}}\right)+\frac{\tilde {\bar\tau}^1_{2n}}{\tilde {{\bar\tau}}_{2n}}\left(\frac{\tilde {\hat \tau}^1_{2n+1}}{\tilde{ \hat \tau}_{2n+1}}-\frac{\tilde \tau^1_{2n+1}}{\tilde \tau_{2n+1}}\right)+\frac{\hat {\bar\tau}^1_{2n}}{\hat {{\bar\tau}}_{2n}}\left(\frac{ {\hat \tau}^1_{2n+1}}{{ \hat \tau}_{2n+1}}-\frac{\hat{\tilde \tau}^1_{2n+1}}{\hat{\tilde \tau}_{2n+1}}\right)
		=	\frac{1}{\tilde{d}}\left(
		\frac{\hat {\bar{\bar\tau}}_{2n}}{\hat {\bar\tau}_{2n+1}}\frac{\hat{\tilde{\tau}}_{2n+1}}	{\hat{\tilde{\bar\tau}}_{2n}}-	\frac{ {\bar{\bar\tau}}_{2n}}{ {\bar\tau}_{2n+1}}	\frac{\tilde \tau_{2n+1}}{\tilde {\bar\tau}_{2n}}\right),
		\\
		&	\frac{1}{\tilde{d}}	\frac{\hat {\bar\tau}_{2n+1}}{\hat {\bar\tau}_{2n}}\frac{\hat{\tilde{\tau}}_{2n}}	{\hat{\tilde{\tau}}_{2n+1}}
		\left(\frac{ {\hat {\bar\tau}}^1_{2n}}{{ \hat {\bar\tau}}_{2n}}-\frac{{\bar\tau}^1_{2n}}{{\bar\tau}_{2n}}\right)+\frac{1}{\hat{d}}
		\frac{ \tilde {\bar\tau}_{2n+1}}{ \tilde{\bar\tau}_{2n}}	\frac{\tilde {\hat \tau}_{2n}}{\tilde {\hat \tau}_{2n+1}}	\left(\frac{{\bar\tau}^1_{2n}}{{\bar\tau}_{2n}}-\frac{ {\tilde {\bar \tau}}^1_{2n}}{{ \tilde {\bar\tau}}_{2n}}\right)+\frac{1}{\tilde{d}}
		\frac{ {\bar{\bar\tau}}_{2n}}{ {\bar\tau}_{2n+1}}	\frac{\tilde \tau_{2n+1}}{\tilde {\bar\tau}_{2n}}\left(\frac{\tilde \tau^1_{2n}}{\tilde \tau_{2n}}-\frac{\tilde {\hat \tau}^1_{2n}}{\tilde{ \hat \tau}_{2n}}\right)=0,\\ &
		\frac{1}{\hat{d}}\frac{ {\bar{\bar\tau}}_{2n+1}}{ {\bar{\bar\tau}}_{2n}}\frac{\hat{\bar{\tau}}_{2n}}	{\hat{{\bar\tau}}_{2n+1}}\left(\frac{ {\hat \tau}^1_{2n+1}}{{ \hat \tau}_{2n+1}}-\frac{\hat{\tilde \tau}^1_{2n+1}}{\hat{\tilde \tau}_{2n+1}}\right)+\frac{1}{\tilde{d}}\frac{ {\bar{\bar\tau}}_{2n+1}}{ {\bar{\bar\tau}}_{2n}}\frac{\tilde{\bar{\tau}}_{2n}}	{\tilde{{\bar\tau}}_{2n+1}}\left(\frac{\tilde {\hat \tau}^1_{2n+1}}{\tilde{ \hat \tau}_{2n+1}}-\frac{\tilde \tau^1_{2n+1}}{\tilde \tau_{2n+1}}\right)+\frac{1}{\tilde{d}}\frac{ \tilde {\bar\tau}_{2n+1}}{ \tilde{\bar\tau}_{2n}}	\frac{\tilde {\hat \tau}_{2n}}{\tilde {\hat \tau}_{2n+1}}\left(\frac{ { \bar\tau}^1_{2n+1}}{{   \bar\tau}_{2n+1}}-\frac{\hat{ \bar\tau}^1_{2n+1}}{\hat{ \bar\tau}_{2n+1}}\right)=0.
		%		\label{compsr:p2d2_tau}
		%	&	\bar {f}(\hat{g}-\tilde{g})+\tilde{\bar{f}}(\tilde{g}-\tilde{\hat{g}})+\hat{\bar {f}}(\hat{\tilde{g}}-\hat{g})= \frac{1}{\tilde{d}}\left(\frac{\bar {v}}{\tilde{v}}-\frac{\hat{\bar {v}}}{\hat{\tilde{v}}}\right),\\ &\label{compsr:p1d3_tau}
		%		\frac{1}{\tilde{d}}\frac{\hat{\bar {u}}}{\hat{\tilde{u}}}(\hat{\bar {f}}-\bar {f})+\frac{1}{\hat{d}}\frac{\tilde{\bar {u}}}{\tilde{\hat{u}}}(\bar {f}-\tilde{\bar {f}})+\frac{1}{\tilde{d}}\frac{\bar {v}}{\tilde{v}}(\tilde{f}-\tilde{\hat{f}})=0,\\ 
		%		&\label{compsr:p2d3_tau}
		%		\frac{1}{\hat{d}}\frac{\bar {\bar {u}}}{\hat{\bar {u}}}(\hat{g}-\hat{\tilde{g}})+\frac{1}{\tilde{d}}\frac{\bar {\bar {u}}}{\tilde{\bar {u}}}(\tilde{\hat{g}}-\tilde{g})+\frac{1}{\tilde{d}}\frac{\hat{\bar {v}}}{\hat{\tilde{v}}}(\bar {g}-\hat{\bar {g}})=0.
	\end{align*}
	Notice that we have large families of solutions  to this system of equations in terms of double Wronskians of  generalized hypergeometric functions.
\end{teo}
\begin{proof}
		Use the $\tau$-function parametrization
	\begin{align*}
		u&=\frac{\tau_{2n+1}}{\tau_{2n}},& 	v&=\frac{\tau_{2n+2}}{\tau_{2n+1}},& f&=\frac{\tau^1_{2n}}{\tau_{2n}},  & g&=\frac{\tau^1_{2n+1}}{\tau_{2n+1}}.
	\end{align*}
\end{proof}

\begin{lemma}\label{lem:tecnical_typeI}
If a matrix $M$ is such that $MX^{(1)}(z)=MX^{(2)}(z)=0$ then $M=0$.
\end{lemma}
\begin{proof}
	We have that 
	\begin{align*}
		M\frac{1}{n!}\frac{\d ^n X^{(1)}(z)}{\d z^n}\Big|_{z=0}=&0, & 	M\frac{1}{n!}\frac{\d ^n X^{(2)}(z)}{\d z^n}\Big|_{z=0}=&0.
	\end{align*}
Observe  that $\frac{1}{n!}\frac{\d ^n X^{(1)}(z)}{\d z^n}\big|_{z=0}$ is the vector with all its entries equal to zero but for a $1$ in the $2n$-th entry.
Similarly, $\frac{1}{n!}\frac{\d ^n X^{(2)}(z)}{\d z^n}\big|_{z=0}$ is the vector with all its entries equal to zero but for a $1$ in the $(2n+1)$-th entry. Therefore, all the columns of $M$ have $0$ all its entries.
\end{proof}
 \begin{lemma}
The following equations are satisfied
 \begin{align}\label{compJo}
 	\hat{T}^\intercal\omega^{(r)}=\omega^{(r)}T^\intercal \\ 
        \label{compJoq}\check{T}^\intercal\omega^{(q)}=\omega^{(q)}T^\intercal \\  
        \label{compJos}\tilde{T}^\intercal\omega^{(s)}=\omega^{(s)}T^\intercal
  \end{align} 
 holds.
 \end{lemma}
 \begin{proof}From eigenvalue equation for $T^\intercal$ and $A^{(a)}$, ${A^{(a)}}^\intercal T =z{A^{(a)}}^\intercal $, and \eqref{con:io1} for this case, $\hat{A}^{(a)}(z)=\frac{\omega^{(r)}}{z+\hat{d}}A^{(a)}(z)$ it can be seen 
 \begin{align*}
%{\hat{A}^{(a)}}^\intercal  
{{}{\hat A}^{(a)}}^\intercal	\hat{T}=z{{}{\hat A}^{(a)}}^\intercal,
 \end{align*} 
so that
 \begin{align*}
 	\hat{T}^\intercal\frac{\omega^{(r)}}{z+\hat{d}}A^{(a)}(z)=z\frac{\omega^{(r)}}{z+\hat{d}}A^{(a)}(z)=\frac{\omega^{(r)}}{z+\hat{d}}T^\intercal A^{(a)}.
  \end{align*} 
Hence, if we denote $M=(\hat{T}^\intercal \omega^{(r)}-\hat{T}^\intercal\omega^{(r)}T)H^{-1}\tilde S$ we find $MX^{(a)}=0$ so that, according to Lemma \ref{lem:tecnical_typeI}, we get that $M=0$ and as $H^{-1}\tilde S$ is a lower triangular invertible matrix we conclude the result.
 and finally we get
 \begin{align*}
       \hat{T}^\intercal\omega^{(r)}=\omega^{(r)}T^\intercal. \end{align*}
 For equations \eqref{compJoq} and \eqref{compJos} it can be proven analogously.
 \end{proof}

 \begin{teo}The six  functions  $u, v, f, g, F $ and $G$ fulfill the following system of six 2D nonlinear  difference equations 
 
\begin{enumerate}
	\item 	 For the shifts {}$ \tilde{}$ and {} $\bar{}$:
\begin{subequations}\label{comps}	\begin{align}\label{comps:p1d1}
 \tilde{d}(G-\bar{F}+\tilde{\bar{F}}-\tilde{G}+\tilde{\bar{f}}(\bar{g}-\bar{\tilde{g}})+\tilde{g}(\tilde{\bar{f}}-\bar{f}))&=\frac{\bar{v}}{\tilde{v}}-\frac{\bar{\bar{u}}}{\tilde{\bar{u}}},\\
 \label{comps:p2d1}
\tilde{d}(F-G+\tilde{G}-\tilde{F}+\tilde{g}(\bar{f}-\tilde{\bar{f}})+\tilde{f}(\tilde{g}-g))&=\frac{\bar{u}}{\tilde{u}}-\frac{\bar{v}}{\tilde{v}},\\ \label{comps:p1d2}
 \frac{\bar{\bar{u}}}{\tilde{\bar{u}}}(\tilde{g}-\tilde{\bar{f}})+\tilde{d}(\bar{g}-\tilde{\bar{g}})(\tilde{G}-\tilde{\bar{F}}-\tilde{\bar{f}}(\tilde{\bar{f}}-\tilde{\bar{g}}))+\tilde{d}\frac{\tilde{\bar{\bar{u}}}}{\tilde{\bar{u}}}&=\tilde{d}\frac{\bar{\bar{u}}}{\bar{u}}+\frac{\bar{\bar{u}}}{\tilde{\bar{u}}}(\bar{g}-\bar{\bar{f}})+\tilde{d}(\bar{f}-\tilde{\bar{f}})(\bar{F}-\bar{G}-\bar{g}(\bar{g}-\bar{\bar{f}})),\\ \label{comps:p2d2}
  \frac{\bar{v}}{\tilde{v}}(\tilde{f}-\tilde{g})+\tilde{d}(\bar{f}-\tilde{\bar{f}})(\tilde{F}-\tilde{G}-\tilde{g}(\tilde{g}-\tilde{\bar{f}}))+\tilde{d}\frac{\tilde{\bar{v}}}{\tilde{v}}&=\tilde{d}\frac{\bar{v}}{v}+\frac{\bar{v}}{\tilde{v}}(\bar{f}-\bar{g})+\tilde{d}(g-\tilde{g})(G-\bar{F}-\bar{f}(\bar{f}-\bar{g})), \\ \label{comps:p1d3} \frac{\bar{\bar{v}}}{\tilde{\bar{v}}}(\tilde{G}-\tilde{\bar{F}}-\tilde{\bar{f}}(\tilde{\bar{f}}-\tilde{\bar{g}}))+\tilde{d}\frac{\tilde{\bar{\bar{u}}}}{\tilde{\bar{u}}}(\bar{f}-\tilde{\bar{f}})&=\tilde{d}\frac{\bar{\bar{v}}}{\bar{v}}(\bar{f}-\tilde{\bar{f}})+\frac{\bar{\bar{u}}}{\tilde{\bar{u}}}(\bar{G}-\bar{\bar{F}}-\bar{\bar{f}}(\bar{\bar{f}}-\bar{\bar{g}})), \\ \label{comps:p2d3}
  \frac{\bar{\bar{u}}}{\tilde{\bar{u}}}(\tilde{F}-\tilde{G}-\tilde{g}(\tilde{g}-\tilde{\bar{f}}))+\tilde{d}\frac{\tilde{\bar{v}}}{\tilde{v}}(\bar{g}-\tilde{\bar{g}})&=\tilde{d}\frac{\bar{\bar{u}}}{\bar{u}}(g-\tilde{g})+\frac{\bar{v}}{\tilde{v}}(\bar{F}-\bar{G}-\bar{g}(\bar{g}-\bar{\bar{f}})).
\end{align}
\end{subequations}
\item For the shifts {}$\check{}$ and {} $\bar{}$:
\begin{subequations}\label{compq-}	
	 \begin{align}\label{compq:p1d1}
		\frac{\bar{v}}{\check{v}}&=\check{d}(G-\bar{F}-\check{G}+\check{\bar{F}}+\check{\bar{f}}(\bar{g}-\check{\bar{g}})+\check{g}(\check{\bar{f}}-\bar{f})),\\
		\label{compq:p2d1} 
		\check{d}(F-G-\check{F}+\check{G}+\check{g}(\bar{f}-\check{\bar{f}})+\check{f}(\check{g}-g)&=-\frac{\bar{v}}{\check{v}},\\ \label{compq:p1d2}\frac{\check{\bar{\bar{u}}}}{\check{\bar{u}}}+(\bar{g}-\check{\bar{g}})(\check{G}-\check{\bar{F}}-\check{\bar{f}}(\check{\bar{f}}-\check{\bar{g}}))&=\frac{\bar{\bar{u}}}{\bar{u}}+(\bar{f}-\check{\bar{f}})(\bar{F}-\bar{G}-\bar{g}(\bar{g}-\bar{\bar{f}})),\\ \label{compq:p2d2}
		\frac{\bar{v}}{\check{v}}(\check{f}-\check{g})+\check{d}(\bar{f}-\check{\bar{f}})(\check{F}-\check{G}-\check{g}(\check{g}-\check{\bar{f}}))+\check{d}\frac{\check{\bar{v}}}{\check{v}}&=\check{d}\frac{\bar{v}}{v}+\check{d}(g-\check{g})(G-\bar{F}-\bar{f}(\bar{f}-\bar{g}))+\frac{\bar{v}}{\check{v}}(\bar{f}-\bar{g}),\\ \label{compq:p1d3} \frac{\bar{\bar{v}}}{\check{\bar{v}}}(\check{G}-\check{\bar{F}}-\check{\bar{f}}(\check{\bar{f}}-\check{\bar{g}}))+\check{d}\frac{\check{\bar{\bar{u}}}}{\check{\bar{u}}}(\bar{\bar{f}}-\check{\bar{\bar{f}}})&=\check{d}\frac{\bar{\bar{v}}}{\bar{v}}(\bar{f}-\check{\bar{f}}), \\ \label{compq:p2d3} \check{d}\frac{\check{\bar{v}}}{\check{v}}(\bar{g}-\check{\bar{g}})&=\check{d}\frac{\bar{\bar{u}}}{\bar{u}}(g-\check{g})+\frac{\bar{v}}{\check{v}}(\bar{F}-\bar{G}-\bar{g}(\bar{g}-\bar{\bar{f}}))  .\end{align}
	\end{subequations}
\item For the shifts {}$ \hat{}$ and {} $\bar{}$:
\begin{subequations}\label{compr-}
	   \begin{align}
		\label{compr:p1d1}
		\hat{d}(G-\bar{F}-\hat{G}+\hat{\bar{F}}+\hat{g}(\hat{\bar{f}}-\bar{f})+\hat{\bar{f}}(\bar{g}-\hat{\bar{g}}))&=-\frac{\bar{\bar{u}}}{\hat{\bar{u}}},\\
		\label{compr:p2d1}
		\hat{d}(F-G-\hat{\bar{F}}+\hat{G}+\hat{f}(\hat{g}-g)+\hat{g}(\bar{f}-\hat{\bar{f}})&=\frac{\bar{u}}{\hat{u}},\\
		\label{compr:p1d2}
		\hat{d}\left(\frac{\hat{\bar{\bar{u}}}}{\hat{\bar{u}}}+(\bar{g}-\hat{\bar{g}})(\hat{G}-\hat{\bar{F}}-\hat{\bar{f}}(\hat{\bar{f}}-\hat{\bar{g}}))\right)+\frac{\bar{\bar{u}}}{\hat{\bar{u}}}(\hat{g}-\hat{\bar{f}})&=\frac{\bar{\bar{u}}}{\hat{\bar{u}}}(\bar{g}-\bar{\bar{f}})+\hat{d}(\frac{\bar{\bar{u}}}{\bar{u}}+(\bar{f}-\hat{\bar{f}})(\bar{F}-\bar{G}-\bar{g}(\bar{g}-\bar{\bar{f}}))),\\
		\label{compr:p2d2}
		\frac{\hat{\bar{v}}}{\hat{v}}+(\bar{f}-\hat{\bar{f}})(\hat{F}-\hat{G}-\hat{g}(\hat{g}-\hat{\bar{f}}))&=\frac{\bar{v}}{v}+(g-\hat{g})(G-\bar{F}-\bar{f}(\bar{f}-\bar{g})),\\
         \label{compr:p1d3}
         \hat{d}\frac{\hat{\bar{u}}}{\hat{u}}(\bar{f}-\hat{\bar{f}})&=\hat{d}\frac{\bar{v}}{v}(f-\hat{f})+\frac{\bar{u}}{\hat{u}}(G-\bar{F}-\bar{f}(\bar{f}-\bar{g})), \\
         \label{compr:p2d3}
         \frac{\bar{v}}{\hat{v}}(\hat{F}-\hat{G}-\hat{g}(\hat{g}-\hat{\bar{f}}))+\hat{d}\frac{\hat{\bar{v}}}{\hat{v}}(\bar{g}-\hat{\bar{g}})&=\hat{d}\frac{\bar{\bar{u}}}{\bar{u}}(g-\hat{g}).
	\end{align}
\end{subequations}
\end{enumerate}
 \end{teo}
 \begin{proof}From \eqref{compJos} we obtained three non-null diagonals. From first diagonal 
 \begin{align*}
     \tilde{d}(\mathfrak a_{+}S^{[2]}-S^{[2]}+\tilde{S}^{[2]}-\mathfrak a_{+}\tilde{S}^{[2]}+\tilde{S}^{[1]}(\mathfrak a_{-}S^{[1]}-\mathfrak a_{-}\tilde{S}^{[1]})+\mathfrak a_{+}\tilde{S}^{[1]}(\tilde{S}^{[1]}-S^{[1]}))=\frac{\mathfrak a_{-}H}{\mathfrak a_{+}\tilde{H}}-\frac{\mathfrak a_{-}^2H}{\tilde{H}},
 \end{align*}
 if we take $I^{(1)}$ part we obtain \eqref{comps:p1d1} and from $I^{(2)}$ we obtain \eqref{comps:p2d1};
 from second diagonal
\begin{multline*}
    (\mathfrak a_{+}\tilde{S}^{[1]}-\tilde{S}^{[1]})\frac{\mathfrak a_{-}^2H}{\tilde{H}}+\tilde{d}(\mathfrak a_{-}S^{[1]}-\mathfrak a_{-}\tilde{S}^{[1]})(\mathfrak a_{+}\tilde{S}^{[2]}-\tilde{S}^{[2]}-\tilde{S}^{[1]}(\tilde{S}^{[1]}-\mathfrak a_{-}\tilde{S}^{[1]}))+\tilde{d}\frac{\mathfrak a_{-}^2\tilde{H}}{\tilde{H}}\\
   =\tilde{d}\frac{\mathfrak a_{-}^2H}{H}+
   \frac{\mathfrak a_{-}^2 H}{\tilde{H}}(\mathfrak a_{-}S^{[1]}
   -\mathfrak a_{-}^2S^{[1]})+\tilde{d}(S^{[1]}-\tilde{S}^{[1]})(S^{[2]}-\mathfrak a_{-}S^{[2]}-\mathfrak a_{-}S^{[1]}(\mathfrak a_{-}S^{[1]}-\mathfrak a_{-}^2S^{[1]})),
\end{multline*}
from part coming from  $I^{(1)}$ it can be obtained \eqref{comps:p1d2} and from part of $I^{(2)}$ it can be obtained \eqref{comps:p2d2}; from third superdiagonal
\begin{multline*}
   \frac{\mathfrak a^3_{-}H}{\mathfrak a_{-}\tilde{H}}(\mathfrak a_{+}\tilde{S}^{[2]}-\tilde{S}^{[2]}-\tilde{S}^{[1]}(\tilde{S}^{[1]}-\mathfrak a_{-}\tilde{S}^{[1]}))+\tilde{d}\frac{\mathfrak a_{-}^2\tilde{H}}{\tilde{H}}(\mathfrak a_{-}^2S^{[1]}-\mathfrak a_{-}^2\tilde{S}^{[1]})=\tilde{d}\frac{\mathfrak a_{-}^3H}{\mathfrak a_{-}H}(S^{[1]}-\tilde{S}^{[1]})+\\+\frac{\mathfrak a_{-}^2H}{\tilde{H}}(\mathfrak a_{-}S^{[2]}-\mathfrak a_{-}^2S^{[2]}-\mathfrak a_{-}^2S^{[1]}(\mathfrak a_{-}^2S^{[1]}-\mathfrak a_{-}^3S^{[1]})),
\end{multline*}
 from part of $I^{(1)}$ it can be obtained \eqref{comps:p1d3} and from part of $I^{(2)}$ it can be obtained \eqref{comps:p2d3}. 
From \eqref{compJoq} we can obtain three non-null diagonals. From first superdiagonal
\begin{equation*}
    I^{(1)}\frac{\mathfrak a_{-}H}{\mathfrak a_{+}\check{H}}-I^{(2)}\frac{\mathfrak a_{-}^2H}{\check{H}}=\check{d}(\mathfrak a_{+}S^{[2]}-S^{[2]}-\mathfrak a_{+}\check{S}^{[2]}+\check{S}^{[2]}+\check{S}^{[1]}(\mathfrak a_{-}S^{[1]}-\mathfrak a_{-}\check{S}^{[1]})+\mathfrak a_{+}\check{S}^{[1]}(\check{S}^{[1]}-S^{[1]}))
\end{equation*}
from part of $I^{(1)}$ it can be obtained \eqref{compq:p1d1} and from part of $I^{(2)}$ it can be obtained \eqref{compq:p2d1}; from second superdiagonal
\begin{multline*}
    I^{(2)}\frac{\mathfrak a_{-}^2H}{\check{H}}(\mathfrak a_{+}\check{S}^{[1]}-\check{S}^{[1]})+\check{d}(\mathfrak a_{-}S^{[1]}-\mathfrak a_{-}\check{S}^{[1]})(\mathfrak a_{+}\check{S}^{[2]}-\check{S}^{[2]}-\check{S}^{[1]}(\check{S}^{[1]}-\mathfrak a_{-}\check{S}^{[1]}))+\check{d}\frac{\mathfrak a_{-}^2\check{H}}{\check{H}}\\=\check{d}\frac{\mathfrak a_{-}^2H}{H}+\check{d}(S^{[1]}-\check{S}^{[1]})(S^{[2]}-\mathfrak a_{-}{S}^{[2]}-\mathfrak a_{-}{S}^{[1]}(\mathfrak a_{-}{S}^{[1]}-\mathfrak a_{-}^2{S}^{[1]}))+I^{(2)}\frac{\mathfrak a_{-}^2H}{\check{H}}(\mathfrak a_{-}S^{[1]}-\mathfrak a_{-}^2S^{[1]})
\end{multline*}
from part of $I^{(1)}$ it can be obtained \eqref{compq:p1d2} and from part of $I^{(2)}$ it can be obtained \eqref{compq:p2d2}; and from third superdiagonal 
\begin{multline*}
    I^{(1)}\frac{\mathfrak a_{-}^3H}{\mathfrak a_{-}\check{H}}(\mathfrak a_{+}\check{S}^{[2]}-\check{S}^{[2]}-\check{S}^{[1]}(\check{S}^{[1]}-\mathfrak a_{-}\check{S}^{[1]}))+\check{d}\frac{\mathfrak a_{-}^2 \check{H}}{\check{H}}(\mathfrak a_{-}^2S^{[1]}-\mathfrak a_{-}^2\check{S}^{[1]})=\\=\check{d}(S^{[1]}-\check{S}^{[1]})\frac{\mathfrak a_{-}^3H}{\mathfrak a_{-}H}+I^{(2)}\frac{\mathfrak a_{-}^2H}{\check{H}}(\mathfrak a_{-}S^{[2]}-\mathfrak a_{-}^2S^{[2]]}-\mathfrak a_{-}^2S^{[1]}](\mathfrak a_{-}^2 S^{[1]}-\mathfrak a_{-}^3S^{[1]}))
\end{multline*}
from part of $I^{(1)}$ we can obtain \eqref{compq:p1d3} and from $I^{(2)}$ we can obtain \eqref{compq:p2d3}.
\end{proof}
\begin{rem}
	Notice that 
\begin{align*}
	F&=\frac{\tau^2_{2n}}{\tau_{2n}}, & 	G&=\frac{\tau^2_{2n+1}}{\tau_{2n+1}}, 
\end{align*}
As we did in Proposition \ref{rem:tau_Nijhoff_Capel} for the multiple Nijhoff--Capel Toda system we can get $\tau$-function representation of the system \eqref{comps}.
\end{rem}

\section{Laguerre--Freud equations for generalized multiple Charlier and Meixner II}
We will begin by applying the previous developments to the multiple orthogonal polynomials discussed in \cite{Arvesu}, specifically the multiple Charlier and Meixner II sequences of multiple orthogonal polynomials. It has been proven in \cite{Arvesu} that in all these cases, the system is AT, indicating that we are dealing with perfect systems. Subsequently, we will proceed with the study of generalizations of these two cases. The original polynomials were found in \cite{charlier} and \cite{meixner}, respectively.

We will demonstrate that it is possible to discuss Laguerre--Freud equations in all these examples. In this multiple scenario, we extend the concept of Laguerre--Freud equations to encompass nonlinear equations satisfied by the coefficients of the recursion relation in the form:
\begin{subequations}\label{eq:LF}
	\begin{align}
\label{eq:LF_gamma}	\gamma_n&=\mathscr F_n(\gamma_{n-1},\beta_{n-1},\alpha_{n-1},\dots),\\ 
\label{eq:LF_beta}		\beta_n&=\mathscr G_n(\gamma_n,\gamma_{n-1},\beta_{n-1}, \alpha_{n-1}, \dots), \\ 
\label{eq:LF_alpha}		\alpha_n&=\mathscr H_n(\gamma_n,\beta_n,\gamma_{n-1},\beta_{n-1},\alpha_{n-1},\dots).
\end{align}
\end{subequations}
\subsection{Charlier multiple orthogonal polynomials}
In this case we set
\begin{align*}w^{(1)}(k)&=\frac{{\eta^{(1)}}^k}{k!}, &w^{(2)}(k)&=\frac{{\eta^{(2)}}^k}{k!},\end{align*}
if ${\eta^{(1)}},{\eta^{(2)}}>0$, it is an AT system,
and polynomials $\theta$, $\sigma^{(1)}$ and $\sigma^{(2)}$ are:
\begin{align*}
	\theta(k)&=k, &\sigma^{(1)}(k)&={\eta^{(1)}}, &\sigma^{(2)}(k)&={\eta^{(2)}}.
\end{align*}
Now, $\sigma^{(1)}$ and $\sigma^{(2)}$ are degree zero polynomials and $\theta$ is a degree  one polynomial, consequently there are not non-zero subdiagonals and there is only one non-zero superdiagonal, this is:
\begin{align*}
	\Psi=\psi^{(0)}+\psi^{(1)}.
\end{align*}
To obtain Laguerre-Freud matrix firstly it is used $\Psi=\sigma^{(1)}(T^\intercal_1)\Pi^\intercal_1+\sigma^{(2)}(T^\intercal_2)\Pi^\intercal_2$, from this expression it is obtained:
\begin{align*}
	\psi^{(0)}&=I^{(1)}{\eta^{(1)}}+I^{(2)}{\eta^{(2)}}, &  \psi^{(1)}&={\eta^{(1)}}D_1+{\eta^{(2)}}D_2.
\end{align*}
From the expression $\Psi=\Pi^{-1}\theta(T)$ we get
\begin{align*}
\psi^{(1)}&=I, & \psi^{(0)}&=\alpha-\mathfrak a_{+}D.
\end{align*}
Attending to the above results   the Laguerre-Freud structure matrix is
\begin{align*}
	\Psi={\eta^{(1)}} I^{(1)}+{\eta^{(2)}} I^{(2)}+\Lambda.
\end{align*}
Let us discuss  $[\Psi,T]=\Psi$. First we obtain
that  $[\Psi,T]=\Lambda^\intercal\tilde{\psi}^{(-1)}+\tilde{\psi}^{(0)}+\tilde{\psi}^{(1)}\Lambda$, i.e.,
\begin{align*}
	\tilde{\psi}^{(-1)}&=\gamma-\mathfrak a_{+}\gamma-\beta({\eta^{(1)}}-{\eta^{(2)}})(I^{(1)}-I^{(2)}), \\ \tilde{\psi}^{(0)}&=\beta-\mathfrak a_{+}\beta, \\ \tilde{\psi}^{(1)}&=\mathfrak a_{-}\alpha-\alpha+({\eta^{(1)}}-{\eta^{(2)}})(I^{(1)}-I^{(2)}), \end{align*}
equating  these to Laguerre--Freud matrix $\Psi$ we obtain three non-trivial equations:
\begin{align*}\mathfrak a_{-}\alpha-\alpha&=I+({\eta^{(1)}}-{\eta^{(2)}})(I^{(2)}-I^{(1)}), \\ \beta-\mathfrak a_{+}\beta&={\eta^{(1)}} I^{(1)}+{\eta^{(2)}}I^{(2)}, \\ \gamma-\mathfrak a_{+}\gamma&=\beta({\eta^{(1)}}-{\eta^{(2)}})(I^{(1)}-I^{(2)}).\end{align*}
From compatibility equations we get
\begin{align*}
	\alpha_{2n}&=2n+\alpha_0, &\alpha_{2n+1}&=2n+1+{\eta^{(2)}}-{\eta^{(1)}}+\alpha_0, \\
\beta_{2n}&=n({\eta^{(1)}}+{\eta^{(2)}}), &\beta_{2n+1}&=n(({\eta^{(1)}}+{\eta^{(2)}})+{\eta^{(1)}},\\ \gamma_{2n+1}&=n{\eta^{(2)}}({\eta^{(2)}}-{\eta^{(1)}}), & \gamma_{2n+2}&=(n+1){\eta^{(1)}}({\eta^{(1)}}-{\eta^{(2)}}), 
\end{align*}
where $n \in \lbrace 0,1,2,\dots \rbrace$. These results have been found previously in  \cite{Arvesu}.

\subsection{Meixner of second kind multiple orthogonal polynomials}
In this case we have:
\begin{align*}\theta(k)&=k,&\sigma^{(1)}(k)&=\eta(k+b_1),&\sigma^{(2)}(k)&=\eta(k+b_2).\end{align*}
The weights are:
\begin{align*}w^{(1)}(k)&=(b_1)_k\frac{\eta^k}{k!},  &w^{(2)}(k)&=(b_2)_k\frac{\eta^k}{k!},\end{align*}
if $b_1,b_2,\eta>0$, according with Lemma \ref{lemma:ATSystems}, it is an AT system.\\
The Laguerre-Freud matrix has four non-trivial diagonals:
\begin{equation*}\Psi=(\Lambda^\intercal)^2\psi^{(-2)}+\dots+\psi^{(1)}\Lambda.\end{equation*}
The main diagonal and the first super diagonal can be obtained by means of \(\Psi=\Pi^{-1}\theta(T)\):
\begin{align*}\psi^{(1)}&=I,& \psi^{(0)}&=\alpha-\mathfrak a_{+}D.\end{align*}
Both subdiagonals are obtained through \(\Psi=\sigma^{(1)}(T_1^\intercal)\Pi_1^\intercal+\sigma^{(2)}(T^\intercal_2)\Pi_2^\intercal\):
\begin{align*}\psi^{(0)}&=\eta(\alpha+\mathfrak a_{+}^2D_1^{[1]}I^{(1)}+\mathfrak a_{+}^2D_2^{[1]}I^{(2)}+b_1I^{(1)}+b_2I^{(2)}), & \psi^{(-1)}&=\eta\beta,&\psi^{(-2)}&=\eta\gamma.\end{align*}
Using Compatibility I we can obtain three non trivial equations from first superdiagonal, first subdiagonals and the main diagonal, respectively:
\begin{align*}\mathfrak a_{-}\alpha-\alpha&=\frac{1}{1-\eta}(I+\eta(I^{(2)}+(b_1-b_2)(I^{(2)}-I^{(1)})),\\ \mathfrak a_{+}\gamma-\gamma&=\frac{\eta}{\eta-1}\beta(I^{(1)}+(b_1-b_2)(I^{(1)}-I^{(2)})),\\\mathfrak a_{+}\beta-\beta&=\frac{1}{\eta-1}(\alpha-\mathfrak a_{+}D).\end{align*}
Equating both expressions for $\psi^{(0)}$ we obtain expressions for $\alpha$:
\begin{align*}\alpha_{2n}&=\frac{1}{1-\eta}2n+\frac{\eta}{1-\eta}(n+b_1),& \alpha_{2n+1}&=\frac{1}{1-\eta}(2n+1)+\frac{\eta}{1-\eta}(n+b_2)\end{align*}

From compatibility conditions we obtain:
\begin{align*}
 \beta_{2n}&=\frac{\eta}{(1-\eta)^2}(3n^2+n(b_1+b_2-2)), & \beta_{2n+1}&=\frac{\eta}{(1-\eta)^2}(3n^2+n(b_1+b_2+1)+b_1),\\ \gamma_{2n}&=\frac{\eta^2}{(1-\eta)^3}n(n+b_1-1)(n+b_1-b_2), &\gamma_{2n+1}&=\frac{\eta^2}{(1-\eta)^3}n(n+b_2-1)(n+b_2-b_1).\end{align*}

As we can see, results for $\alpha_{2n}, \alpha_{2n+1}, \beta_{2n}, \beta_{2n+1}, \gamma_{2n}$ and $\gamma_{2n+1}$ agree with results obtained in \cite{Arvesu}.
%From Compatibility II we obtain three non-trivial equations:
%\begin{align*}\mathfrak a_{-}^2\alpha-\alpha&=(\eta+2)I, \\ \mathfrak a_{+}\gamma-\gamma&=\beta(\mathfrak a_{-}\alpha-\alpha-1-\eta),\\\mathfrak a_{+}\beta-\beta&=\frac{\eta}{\eta-1}(\alpha+\mathfrak a_{+}^2D_1^{[1]}+\mathfrak a_{+}^2D_2^{[1]}+b_1I^{(1)}+b_2I^{(2)}).\end{align*}
%Laguerre-Freud equations for this case are:
%\begin{align*}
%    \alpha_{n+1}&=\alpha_n+\frac{1}{1-\eta}(1+\eta(\frac{1}{2}(1+(-1)^{n+1})+(b_1-b_2)(-1)^{n+1})), \\ \beta_{n+1}&=\beta_n+\frac{1}{\eta-1}(n-\alpha_n), \\ \gamma_{n+2}&=\gamma_{n+1}+\frac{\eta}{1-\eta}\beta_{n+1}(\frac{1}{2}((-1)^n+1)+(b_1-b_2)(-1)^n).
%\end{align*}

\subsection{Generalized Charlier multiple orthogonal polynomials}
In this case we have that:
\begin{align*}\theta (k)&=k(k+c),& \sigma^{(1)} (k)&=\eta^{(1)}, &\sigma^{(2)}(k)&=\eta^{(2)}.\end{align*}
The weights are 
	\begin{align*}
	w^{(1)}(k)&=
\frac{1}{(c+1)_k}
	\dfrac{{\eta^{(1)}}^k}{k!}, &w^{(2)}(k)&=\frac{1}{(c+1)_k}\dfrac{{\eta^{(2)}}^k}{k!} ,
\end{align*}
are for $c>-1,\eta^{(1)},\eta^{(2)}>0$, according to the Lemma \ref{lemma:ATSystems}, AT systems.
Due to the $\sigma^{(1)}$, $\sigma^{(2)}$ and $\theta$ polynomials degrees, Laguerre-Freud matrix has three non-trivial diagonals:
\begin{align*}\Psi=\psi^{(0)}+\psi^{(1)}\Lambda+\psi^{(2)}\Lambda^2.\end{align*}
By means of \(\Psi=\Pi^{-1}\theta (T)\) we can obtain expressions for first and second subdiagonals and through \(\Psi=\Pi^\intercal_1\sigma^{(1)}(T_1^\intercal)+\Pi^\intercal_2\sigma^{(2)}(T_2^\intercal)\) the main diagonal is obtained:
\begin{align*}\psi^{(0)}&=\eta^{(1)}I^{(1)}+\eta^{(2)}I^{(2)},& \psi^{(1)}&=\mathfrak a_{-}\alpha+\alpha+c-\mathfrak a_{+}D,  & \psi^{(2)}&=I.\end{align*}
From Compatibility I non-trivial equations for first superdiagonal, main diagonal and first subdiagonal are obtained:
\begin{align*}\mathfrak a_{-}\beta-\mathfrak a_{+}\beta&=(\alpha+I-\mathfrak a_{-}\alpha)(\mathfrak a_{-}\alpha+\alpha+c-\mathfrak a_{+}D)+(\eta^{(1)}-{\eta^{(2)}})(I^{(2)}-I^{(1)}), \\
\eta^{(1)}I^{(1)}+{\eta^{(2)}}I^{(2)}&=\gamma-\mathfrak a_{+}^2\gamma+\beta(\mathfrak a_{-}\alpha+\alpha+c-\mathfrak a_{+}D)-\mathfrak a_{+}\beta(\alpha+\mathfrak a_{+}\alpha+c-\mathfrak a_{+}^2D), \\
\beta (\eta^{(1)}-{\eta^{(2)}})(I^{(1)}-I^{(2)}) &=\gamma(\mathfrak a_{-}^2\alpha+\mathfrak a_{-}\alpha+c-D)-\mathfrak a_{+}\gamma(\alpha+\mathfrak a_{+}\alpha+c-\mathfrak a_{+}^2D). 
\end{align*}
These equations can be written element by element respectively as
\begin{align*}
	\beta_{n+2}-\beta_n&=(\alpha_n+1-\alpha_{n+1})(\alpha_{n+1}+\alpha_n+c-n)+(-1)^{n+1}(\eta^{(1)}-{\eta^{(2)}}),\\
\frac{\eta^{(1)}+{\eta^{(2)}}}{2}+(-1)^n\frac{\eta^{(1)}-{\eta^{(2)}}}{2}&=\gamma_{n+2}-\gamma_n+\beta_{n+1}(\alpha_{n+1}+\alpha_n+c-n)-\beta_n(\alpha_n+\alpha_{n-1}+c-n+1),\\
(-1)^n\beta_{n+1}(\eta^{(1)}-{\eta^{(2)}})&=\gamma_{n+2}(\alpha_{n+2}+\alpha_{n+1}+c-n-1)-\gamma_{n+1}(\alpha_n+\alpha_{n-1}+c-n+1).
\end{align*}
From Compatibility II we obtain three non-trivial equations from first superdiagonal, main diagonal and first subdiagonal, respectively:
\begin{align*}\mathfrak a_{-}\beta-\mathfrak a_{+}\beta&=(\alpha+I-\mathfrak a_{-}\alpha)(\mathfrak a_{-}\alpha+\alpha+c-\mathfrak a_{+}D)+(\eta^{(1)}-{\eta^{(2)}})(I^{(2)}-I^{(1)}), \\ \gamma-\mathfrak a_{+}^2\gamma&=\eta^{(1)}I^{(1)}+{\eta^{(2)}} I^{(2)}+\mathfrak a_{+}\beta(\alpha+\mathfrak a_{+}\alpha+c-\mathfrak a_{+}^2D)-\beta(\mathfrak a_{-}\alpha+\alpha+c-\mathfrak a_{+}D), \\ \beta(\eta^{(1)}-{\eta^{(2)}})(I^{(1)}-I^{(2)})&=\gamma(\mathfrak a_{-}^2\alpha+\mathfrak a_{-}\alpha+c-D)-\mathfrak a_{+}\gamma(\alpha+\mathfrak a_{+}\alpha+c-\mathfrak a_{+}^2D) .\end{align*}
\begin{pro}[Laguerre--Freud equations]
	Laguerre-Freud Equations \eqref{eq:LF} for generalized Charlier multiple orthogonal polynomials are
\begin{align*}
    \alpha_{n+2}&=n+1-c-\alpha_{n+1}+\gamma_{n+2}^{-1}\left((-1)^{n}\beta_{n+1}(\eta^{(1)}-{\eta^{(2)}})+\gamma_{n+1}(\alpha_n+\alpha_{n-1}+c-n+1)\right),\\ \beta_{n+2}&=\beta_n+(\alpha_n+1-\alpha_{n+1})(\alpha_{n+1}+\alpha_n+c-n)+(-1)^{n+1}(\eta^{(1)}-{\eta^{(2)}}),\\ \gamma_{n+2}&=\gamma_n+\beta_n(\alpha_n+\alpha_{n-1}+c-n+1)-\beta_{n+1}(\alpha_{n+1}+\alpha_n+c-n)+\frac{\eta^{(1)}+{\eta^{(2)}}}{2}+(-1)^{n}\frac{\eta^{(1)}-{\eta^{(2)}}}{2}.
\end{align*}
\end{pro}

\subsection{Generalized Meixner  of second kind multiple orthogonal polynomials}
In this case we have
\begin{align*}\theta(k)&=k(k+c),&\sigma^{(1)}(k)&=\eta(k+b_1),&\sigma^{(2)}(k)&=\eta(k+b_2).\end{align*}
The weights are:
\begin{align*}w^{(1)}(k)&=\frac{(b_1)_k}{(c+1)_k}\frac{\eta^k}{k!}, &w^{(2)}(k)&=\frac{(b_2)_k}{(c+1)_k}\frac{\eta^k}{k!},\end{align*}
if $b_1, b_2, \eta >0$ and $c>-1$, according with Lemma \ref{lemma:ATSystems}, it is an AT system.\\

In this case Laguerre--Freud matrix has five non-null diagonals, it is:
\begin{equation*}\Psi=(\Lambda^\intercal)^2\psi^{(-2)}+\dots+\psi^{(2)}\Lambda^2.\end{equation*}
Second and first superdiagonals are obtained using 
\(\Psi=\Pi^{-1}\theta(T)\):

\begin{align*}\psi^{(2)}&=I,& \psi^{(1)}&=\alpha+\mathfrak a_{-}\alpha+c-\mathfrak a_{+}D,& \psi^{(0)}&=\mathfrak a_{+}\beta+\beta+\alpha(\alpha+c)-\mathfrak a_{+}D(\mathfrak a_{+}\alpha+\alpha+c)+\mathfrak a_{+}^2\Pi^{[-2]}.\end{align*}
Using \(\Psi=\sigma^{(1)}(T_1^\intercal)\Pi_1^\intercal+\sigma^{(2)}(T^\intercal_2)\Pi_2^\intercal\):
\begin{align*}\psi^{(-2)}&=\eta\gamma,&\psi^{(-1)}&=\eta\beta,&\psi^{(0)}&=\eta(\alpha+I^{(1)}(\mathfrak a_{+}^2D_1^{[1]}+b_1)+I^{(2)}(\mathfrak a_{+}^2D_2^{[1]}+b_2)).\end{align*}
From Compatibility I we get three non-trivial equations from the first subdiagonal, the main diagonal and the first superdiagonal, respectively:
\begin{gather*}
	\eta(\mathfrak a_{+}\gamma-\gamma)+\gamma(\mathfrak a_{-}\alpha+\mathfrak a_{-}^2\alpha+c-D)-\mathfrak a_{+}\gamma(\mathfrak a_{+}\alpha+\alpha+c-\mathfrak a_{+}^2D)=\eta\beta\left(I^{(1)}+(b_1-b_2)(I^{(1)}-I^{(2)})\right),\\
\begin{multlined}[\textwidth]\gamma-\mathfrak a_{+}^2\gamma=\eta\left(\alpha+\beta-\mathfrak a_{+}\beta+I^{(1)}(\mathfrak a_{+}^2D_1^{[1]}+b_1)+I^{(2)}(\mathfrak a_{+}^2D_2^{[1]}+b_2)\right)+\mathfrak a_{+}\beta(\mathfrak a_{+}\alpha+\alpha+c-\mathfrak a_{+}^2D)\\-\beta(\alpha+\mathfrak a_{-}\alpha+c-\mathfrak a_{+}D),\end{multlined}\\
\mathfrak a_{-}\beta-\mathfrak a_{+}\beta+(\mathfrak a_{-}\alpha-\alpha)(\alpha-\mathfrak a_{-}\alpha+c-\mathfrak a_{+}D)+\eta\left(\alpha-\mathfrak a_{-}\alpha+(b_1-b_2)(I^{(1)}-I^{(2)})-I^{(2)}\right)=\alpha+\mathfrak a_{-}\alpha+c-\mathfrak a_{+}D.
\end{gather*}
Equations above can be expressed as
\begin{gather*}
	\begin{multlined}\eta(\gamma_{n+1}-\gamma_{n+2})+\gamma_{n+2}(\alpha_{n+1}+\alpha_{n+2}+c-(n+1))-\gamma_{n+1}(\alpha_{n}+\alpha_{n+1}+c-n)=\eta\beta_{n+1}\left(\frac{1+(-1)^n}{2}+(-1)^n(b_1-b_2)\right),\end{multlined}\\
\begin{multlined}[\textwidth]\gamma_{n+2}-\gamma_n+\beta_{n+1}(\alpha_n+\alpha_{n+1}+c-n)-\beta_n(\alpha_{n-1}+\alpha_n+c-(n-1))+\eta(\beta_n-\beta_{n+1})\\=\eta\left(\alpha_n+\frac{n-1}{2}+\left(\frac{1}{2}+b_2\right)\frac{1-(-1)^n}{2}+\frac{b_1}{2}(1+(-1)^n)\right),\end{multlined} \\
\begin{multlined}
	\alpha_n+\alpha_{n+1}+c-n=\beta_{n+2}-\beta_n+(\alpha_{n+1}-\alpha_n)(\alpha_n-\alpha_{n+1}+c-n)+\eta\left(\alpha_n-\alpha_{n+1}+(-1)^n(b_1-b_2)+\frac{(-1)^n-1}{2}\right),
	\end{multlined}
\end{gather*}
respectively.

From Compatibility II we obtain 
\begin{gather*}\eta(\mathfrak a_{+}\gamma-\gamma)+\gamma(\mathfrak a_{-}\alpha+\mathfrak a_{-}^2\alpha+c-D)-\mathfrak a_{+}\gamma(\mathfrak a_{+}\alpha+\alpha+c-\mathfrak a_{+}^2D)=\eta\beta(I^{(1)}+(b_1-b_2)(I^{(1)}-I^{(2)})),\\
\begin{multlined}[\textwidth]\gamma-\mathfrak a_{+}^2\gamma=\eta[\alpha+\beta-\mathfrak a_{+}\beta+I^{(1)}(\mathfrak a_{+}^2D_1^{[1]}+b_1)+I^{(2)}(\mathfrak a_{+}^2D_2^{[1]}+b_2)]+\mathfrak a_{+}\beta(\mathfrak a_{+}\alpha+\alpha+c-\mathfrak a_{+}^2D)\\-\beta(\alpha+\mathfrak a_{-}\alpha+c-\mathfrak a_{+}D)\end{multlined}\\
\begin{multlined}[\textwidth]\eta(I+\mathfrak a_{-}^2\alpha-\alpha)=T^2_{-}\beta+\mathfrak a_{-}\beta-\beta-\mathfrak a_{+}\beta+\mathfrak a_{-}^2\alpha(\mathfrak a_{-}^2\alpha+c)-\alpha(\alpha+c)-\mathfrak a_{-}D(\mathfrak a_{-}^2\alpha+\mathfrak a_{-}\alpha)+\mathfrak a_{+}D(\alpha+\mathfrak a_{+}\alpha)\\-2cI+\pi^{[-2]}-\mathfrak a_{+}^2\pi^{[-2]}.\end{multlined}
\end{gather*}
as we can see that first and second equations appear in compatibility I.

\begin{pro}[Laguerre--Freud equations]
	Laguerre-Freud equations \eqref{eq:LF} for  generalized Meixner II multiple orthogonal polynomials are
\begin{align*}
    \alpha_{n+2}&=n+1-c-\alpha_{n+1}+\gamma_{n+2}^{-1}\left(\gamma_{n+1}\left(\alpha_n+\alpha_{n+1}+c-n)+\eta(\gamma_{n+2}-\gamma_{n+1})+\eta\beta_{n+1}(\frac{1+(-1)^n}{2}+(-1)^{n}(b_1-b_2)\right)\right),\\
    \beta_{n+2}&=
    	\beta-n+\alpha_n+\alpha_{n+1}+c-n-(\alpha_{n+1}-\alpha_n)(\alpha-n-\alpha_{n+1}+c-n)-\eta\Big(\alpha_n-\alpha_{n+1}+(-1)^n(b_1-b_2)+\frac{(-1)^n-1}{2}\Big), 
   \\
    \gamma_{n+2}&=\begin{multlined}[t][\textwidth]
    	\gamma_n+\beta_n(\alpha_{n-1}+\alpha_n+c-n+1)-\beta_{n+1}(\alpha_n+\alpha_{n+1}+c-n)-\eta(\beta_n-\beta_{n+1})\\
    	+\eta\left(\alpha_n+\frac{n-1}{2}+\left(\frac{1}{2}+b_2\right)\frac{1-(-1)^n}{2}+\frac{b_1}{2}(1+(-1)^n)\right). 
    \end{multlined}
\end{align*}
\end{pro}

\section*{Conclusions and outlook}

Adler and van Moerbeke extensively utilized the Gauss--Borel factorization of the moment matrix in their studies of integrable systems and orthogonal polynomials \cite{adler_moerbeke_1,adler_moerbeke_2,adler_moerbeke_4}. We have also employed these ideas in various contexts, including CMV orthogonal polynomials, matrix orthogonal polynomials, multiple orthogonal polynomials, and multivariate orthogonal polynomials \cite{am,afm,nuestro0,nuestro1,nuestro2,ariznabarreta_manas0,ariznabarreta_manas01,ariznabarreta_manas2,ariznabarreta_manas3,ariznabarreta_manas4,ariznabarreta_manas_toledano}. For a comprehensive overview, refer to \cite{intro}.

In a recent work \cite{Manas_Fernandez-Irisarri}, we applied this approach to investigate the implications of the Pearson equation on the moment matrix and Jacobi matrices. To describe these implications, we introduced a new banded matrix called the Laguerre-Freud structure matrix, which encodes the Laguerre--Freud relations for the recurrence coefficients. We also discovered that the contiguous relations satisfied generalized hypergeometric functions, determining the moments of the weight described by the squared norms of the orthogonal polynomials. This led to a discrete Toda hierarchy known as the Nijhoff--Capel equation \cite{nijhoff}.

In \cite{Manas}, we further examined the role of Christoffel and Geronimus transformations in describing the aforementioned contiguous relations and used Geronimus--Christoffel transformations to characterize shifts in the spectral independent variable of the orthogonal polynomials. Building on this program, \cite{Fernandez-Irrisarri_Manas} delved deeper into the study of discrete semiclassical cases, discovering Laguerre--Freud relations for the recursion coefficients of three types of discrete orthogonal polynomials: generalized Charlier, generalized Meixner, and generalized Hahn of type I.

In \cite{Fernandez-Irrisarri_Manas_2}, we completed this program by obtaining Laguerre-Freud structure matrices and equations for three families of hypergeometric discrete orthogonal polynomials.

In this paper, we extend the previous lines of research to multiple orthogonal polynomials on the step line with two weights. We derive hypergeometric $\tau$-function representations for these multiple orthogonal polynomials and introduce a third-order integrable nonlinear Toda-type equation for these $\tau$-functions. Additionally, we discover systems of nonlinear discrete Nijhoff--Capel Toda-type equations satisfied by these objects. Finally, we apply the Laguerre--Freud matrix structure and compatibilities to derive Laguerre--Freud equations for the multiple generalized Charlier and multiple generalized Meixner of the second type.

As a future outlook, we highlight the following five lines of research:

\begin{enumerate}
	\item Finding larger families of hypergeometric AT systems.
	\item Analyzing the case of $p$ weights with $p>2$.
	\item Studying the connections between the topics discussed in this paper and the transformations presented in \cite{bfm,bfm2} and quadrilateral lattices \cite{quadrilateral1,quadrilateral2}.
	\item Exploring the relationship with multiple versions of discrete and continuous Painlevé equations involving the coefficients $\alpha_n$, $\beta_n$, and $\gamma_n$. The results presented in this paper constitute progress in this direction, as identifying this connection is a nontrivial problem that can be further investigated in future works.
	\item Discussing higher-order integrable flows and finding applications to multicomponent KP-type equations.
\end{enumerate}

\section*{Acknowledgments}
The authors acknowledges Spanish ``Agencia Estatal de Investigación'' research projects [PGC2018-096504-B-C33], \emph{Ortogonalidad y Aproximación: Teoría y Aplicaciones en Física Matemática} and [PID2021- 122154NB-I00], \emph{Ortogonalidad y Aproximación con Aplicaciones en Machine Learning y Teoría de la Probabilidad}.

\section*{Declarations}

\begin{enumerate}
	\item \textbf{Conflict of interest:} The authors declare no conflict of interest.
	\item \textbf{Ethical approval:} Not applicable.
	\item \textbf{Contributions:} All the authors have contribute equally.
	\item \textbf{Data availability:} This paper has no associated data.
\end{enumerate}

\end{document}